\documentclass[12pt,a4paper]{article}
\usepackage[english]{babel}
\usepackage{bm}
\usepackage{amsmath,amsthm,amssymb}
\usepackage{amsthm}
\usepackage{cases}
\usepackage{latexsym}

\usepackage{enumerate}
\usepackage{xcolor}
\usepackage{courier}        
\usepackage{graphicx}

\theoremstyle{plain}
\newtheorem{thm}{Theorem}
\newtheorem{lem}[thm]{Lemma}

\newtheorem{rem}{Remark}
\newtheorem{cor}{Corollary}

\theoremstyle{definition}
\newtheorem{defn}{Definition}[section]

\theoremstyle{remark}

\newcommand{\norm}[1]{\Arrowvert #1 \Arrowvert}

\begin{document}
\title{Pathwise quantitative particle approximation of nonlinear stochastic  Fokker-Planck equations via  relative entropy}

\author{ Christian Olivera \footnote{Departamento de Matem\'atica, Universidade Estadual de Campinas, Brazil. \texttt{colivera@ime.unicamp.br}.}
 \and    Alexandre B. de Souza \footnote{Departamento de Matem\'atica, Universidade Estadual de Campinas, Brazil. \texttt{a265040@dac.unicamp.br}.} \and 
}

\date{}
\maketitle
\begin{abstract}
We derive the nonlinear stochastic Fokker-Planck equation
from  stochastic particle systems   with individual and environmental noises via relative entropy method,  with  pathwise quantitative bounds.  Moreover, we prove the existence of a unique strong solution to the associated  Fokker-Planck equation. Our proof is based on  tools from PDE analysis, stochastic analysis, functional inequalities, and  also we use the dissipation of entropy which provides some  bound on the Fisher information of the particle system. 
The approach applies to repulsive and attractive kernels.
\end{abstract}

\vspace{0.3cm} \noindent {\bf MSC2010 subject classification:} 49N90,  60H30, 60K35
\section{Introduction}

The purpose of this paper is to study the mean-field limit of a stochastic, moderately interacting particle system in order to derive the following  non-linear stochastic Fokker-Planck equation

\begin{align}\label{SPDE_Ito}
\mathrm{d}\rho_t = \frac{1}{2} \Delta \rho_t \,dt + \frac{1}{2}D^2 \rho_t (\sigma \sigma^\top)_t \, dt - \nabla \cdot( \rho_t (K\ast \rho_t)) \, dt  - \nabla \rho_t \cdot \sigma_t\, d  B_t
\end{align}
for $t \in [0,T]$, $T >0$.  Equation  (\ref{SPDE_Ito}) is considered for arbitrary dimensions $d \geq 1$.
Our contribution is to provide a rigorous microscopic derivation of the SPDE of the form (\ref{SPDE_Ito}), supplemented by  pathwise  quantitative convergence estimates   via the  relative entropy method.   The main novelty is to apply the relative entropy directly to the mollified empirical measure   in the setting of moderately interacting particle systems.

Our starting point is the following stochastic moderately interacting particle system  of $N$ indistinguishable particles on  $\mathbb{T}^d$ given by 
    
\begin{align}\label{particles}
d  X_t^{i,N} &=  \frac{1}{N}\sum_{k=1}^{N} \left(K \ast V^{N}\right)\left(X_{t}^{i,N} -X_{t}^{k,N}\right) \, dt +  \,d W_t^{i,N} + \sigma_t\, dB_t  \quad 
\end{align}
	
	\noindent where $W_{t}^{i}$ and $B_{t}$ are  independent standard $\mathbb{T}^d$-valued Brownian motions, defined on a filtered probability space $(\Omega,\mathcal{F},(\mathcal{F}_t)_{t\geq0},\mathbb{P})$, the interaction kernel $V^{N} $ depends on the number of particles  $N \in \mathbb{N} $ via the moderate interaction parameter $\beta$.
    In the interacting particle system (\ref{particles}),  $ \sigma_t\, dB_t$  represents the common environmental noise shared by all particles. Such interacting systems arise naturally in various areas of science and engineering, including statistical mechanics problems in mathematical physics, biology, 
    numerical Monte-Carlo simulations, and various other fields, see \cite{Carmona}, \cite{Carmona2}, \cite{Jabin2}, \cite{Live} and \cite{Lac}.

The microscopic \emph{empirical process} of this $N$-particle system, which is a probability measure on the ambient space $\mathbb{T}^d$,  is given as usual by \begin{equation}S_{t}^{N}\doteq\frac{1}{N} \sum_{i=1}^{N}\delta_{X_{t}^{i,N}}, \qquad t\geqslant 0, \label{eq:mutN}\end{equation} where $\delta_{a}$ is the delta Dirac measure concentrated at $a \in \mathbb{T}^d$. Then, $(S_{t}^{N})_{t\geqslant 0}$ is a  measure-valued process associated to the $\mathbb{T}^d$--valued processes $\{t\mapsto X_{t}^{i,N}\}_{i=1,...,N}$.  

The main goal of this paper is to investigate the large $N$ limit of  the dynamical process
$(S_{t}^{N})_{t\geqslant 0}$. As proven by Sznitman \cite{Sz}, the convergence of the empirical measure towards a constant random variable  $\rho_{t}$ is equivalent to the property of propagation of chaos. For that purpose, we introduce the \emph{mollified empirical measure} \[\rho_t^{N}\doteq V^{N} \ast S_t^{N} = \int_{\mathbb{T}^d} V^N(\cdot-y) \,S_t^N(dy),\] 
which is more regular than $S_t^N$.  Our results  provide  a point-wise in $\omega$ quantitative estimate of the distance between the mollified measure $\rho_t^N$ and the unique solution  $\rho_t$ of  equation \eqref{SPDE_Ito},  which   takes the following  form: 
there exists a constant $C=C_{\omega}>0$ such that for any $N\in\mathbb{N}$,
\begin{align*} 
         \sup_{t \in [0,T]}\mathcal{H}(\rho_t^N|\rho_t) \leq C(\mathcal{H}(\rho_0^N|\rho_0) + N^{-\theta})  ,
\end{align*}
where $\theta$ is an explicit positive parameter, $T>0$ is a time horizon. Then by the classical Csiszar-Kullback-Pinsker inequality, the relative entropy estimate implies the quantitative convergence in $L^{1}(\mathbb{T}^{d})$ sense.  To show that, we use tools from PDE analysis, 
stochastic analysis, functional inequalities , also we use the dissipation of entropy which provides some  bound on the Fisher information of the particle system. Our second contribution is to prove the well-posedness of the nonlinear stochastic Fokker–Planck equation  (\ref{SPDE_Ito}).

\subsection*{Related works}
The mean field limit for the first-order systems, exemplified by $(\ref{particles})$ with $\sigma =0$  has been extensively studied over the last decade, see for example \cite{Bres, Carri,Font,FournierJourdain,FHM, Tardy, Gui,Jabin_2018,Nguyen,  Pisa, Serfaty, Toma}. 

The relative entropy method to prove quantitative propagation of chaos result for McKean-Vlasov systems was first introduced  in  \cite{Jabin_2018} for general first-order systems with $W^{-1, \infty}$ kernels, including the point vortex model approximating the 2D Navier-Stokes equation. Recently much progress has been made in extending the relative entropy method to more general cases and models, especially with singular interacting kernels. Those results include  \cite{Gui} to the uniform in time propagation of chaos by using the logarithmic Sobolev inequality   for the limit density,  and \cite{Carri} for the derivation of the mean-field approximation for Landau-like equations.  In addition, in \cite{Chen}, a combination between the relative entropy and the regularised $L^{2}$-estimate in \cite{Oelschlager87} has been used to prove a propagation of chaos result for the viscous porous medium equation from a moderately interacting particle system.

We also mention the recent work  \cite{Serfaty}  where mean-field limit and propagation of chaos of McKean-Vlasov equations with singular interacting kernels has been considered with the method of modulated energy. Instead of focusing on the  joint law level as in the relative entropy method, this modulated energy method works on the empirical measure of the particle system. For more results in this approach see for example \cite{Cho}, \cite{Cho2}  and \cite{Nguyen}. A successful combination of the techniques in   \cite{Serfaty} and \cite{Jabin_2018}  made it possible to obtain a propagation of chaos result in $L^{1}$-norm for mean-field systems with logarithmic interaction potentials in arbitrary dimensions which includes the Keller-Segel system in dimension 2, see  \cite{Bres}.

Moderately interacting particle systems with regular coefficients and their trajectorial propagation of chaos were studied initially in \cite{Oelschlager84, Oelschlager85,JourdainMeleard, Meleard}.  Based on a mild formulation of the empirical measure of a moderately interacting system and semigroup theory, in  \cite{FlandoliLeimbachOlivera} recently developed a technique to approximate nonlinear PDEs by smoothed empirical measures in strong functional topologies. This technique was also applied for a PDE-ODE system related to aggregation phenomena, see \cite{FlandoliLeocata}; for non-local conservation laws, see  \cite{Simon}; for the 2d Navier-Stokes equation, see \cite{FlandoliOliveraSimon}, etc. About more advances in  moderate particle systems,  see for instance \cite{Ansgar}, \cite{Chen}, \cite{Ansgar2}, \cite{Hao} and  \cite{Simon2}.

The derivation of the  SPDE  (\ref{SPDE_Ito}) from particle systems is interesting and challenging.
  For systems with uniformly Lipschitz interaction coefficients \cite{Cogui} established conditional propagation of chaos.    The entropy method has recently been explored for 
  systems with individual and common noise, as shown in \cite{Shao} for incompressible  the Navier-Stokes equations, \cite{Chen2} for the Hegselmann-Krause model, and \cite{Niko} for mean-field systems with bounded kernels.  Additionally, we mention the result   in \cite{Luo2}, which  is  approximated the stationary solution of the stochastic 2-dimensional Navier-Stokes equation by the point vortex model  with common noise. For more results  , see \cite{Maurelli}, \cite{correa}  and \cite{Kotelenez}.

  We emphasize that compared with the works \cite{Chen},
  \cite{Chen2},  \cite{Niko}, \cite{Shao} we obtain pathwise estimates  of the relative entropy and in the aforementioned articles it is at the level of the joint law of the particle system.  In particular, the advantage of our methodology is that it yields bounds directly at the level of the trajectories of the particle system.
The precise connection between these two approaches is not yet fully understood and will be investigated in future work.

\subsection{Notations}
	For \( d \geq 1 \), let \( C^k(\mathbb{T}^d) \) denote the space of \( k \)-times continuously differentiable functions defined on the \( d \)-dimensional torus \( \mathbb{T}^d \doteq \left[-\frac{1}{2}, \frac{1}{2}\right]^d \), where \( k \in \mathbb{N} \cup \{\infty\} \). The space of probability density functions on \( \mathbb{T}^d \) is denoted by \( \mathcal{P}(\mathbb{T}^d) \).
	
	For a measure space \((X, \mathcal{M}, \mu)\) and a measurable function \( f: X \to \mathbb{R} \) (denoted by \( f \in \mathcal{M} \)), we define the duality pairing between the measure \(\mu\) and the function \(f\) as  
\[
\left\langle \mu, f \right\rangle \doteq \int_{X} f \, d\mu.
\]  
Also, for  $a \in [1,\infty)$  the Lebesgue space is given by
	$$L^a=L^a(X)=\left\{f \in \mathcal{M} \mid\|f\|_{a}\doteq \left(\int_{X}|f|^a d\mu\right)^{\frac{1}{a}}<\infty\right\}
	$$
	and if $a=\infty$ 
	$$L^{\infty}=L^{\infty}(X)=\left\{f \in \mathcal{M} \mid\|f\|_{\infty}\doteq \text{ess} \sup_{x \in X}|f(x)|<\infty\right\}.
	$$
	In some contexts we will write $\|\cdot\|_{a} =\|\cdot\|_{L^a(X)}$.
    
For a measurable space \((X, \mathcal{M})\), (a measure $\mu$ on it is denoted by \( \mu \in \mathcal{M} \)), the space of bounded Radon measures is given by
    $$BV(X) = \left\{\mu \in \mathcal{M}\mid \|\mu\|_{BV} \doteq \sup \left\{\sum_{j=1}^m|\mu(X_j)|\mid m \in \mathbb{N}, X = \cup_{j=1}^m X_j \right\} < \infty\right\}.$$
    
	Let \((U, \|\cdot\|_U)\) be a Banach space. For \(T > 0\), let \(\mathcal{L}\) denote the set of Bochner's measurable functions. The Bochner space \(L^aU\), for \(a \in [1, \infty)\), is defined as  
\[
L^aU = L^a([0, T]; U) = \left\{ f : [0, T] \to U \mid f \in \mathcal{L}, \ \|f\|_{L^aU} \doteq \left( \int_0^T \|f(t)\|_U^a \, dt \right)^{\frac{1}{a}} < \infty \right\}.
\]  

and \(a = \infty\) by
\[
L^\infty U = L^\infty([0, T]; U) = \left\{ f : [0, T] \to U \mid f \in \mathcal{L}, \ \|f\|_{L^\infty U} \doteq \sup_{t \in [0, T]} \|f(t)\|_U < \infty \right\}.
\]  

For the space of tempered distributions on \(\mathbb{T}^d\), we denote it by \(\mathcal{S}'\). For \(q \geq 1\) and \(n \in \mathbb{R}\), let us define the Bessel potential space by
	$$H_{q}^{n}=H_{q}^{n}\left(\mathbb{T}^{d}\right)=\left\{f \in \mathcal{S}' \mid\|f\|_{n,q}\doteq\left\|(I-\Delta)^{\frac{n}{2}} f\right\|_{q}<\infty\right\}.
	$$
	
For $\gamma \in (0,1]$ the H\"older space on $\mathbb{T}^d$ is given by $$C^{\gamma}=C^{\gamma}(\mathbb{T}^d)=\left\{f:\mathbb{T}^d \to \mathbb{R}^{e} \mid\|f\|_{\gamma}\doteq \|f\|_{\infty} + \sup_{x,y\in \mathbb{T}^d}\frac{|f(x)-f(y)|}{|x-y|^{\gamma}}<\infty\right\}.$$

	 Let a filtered probability space $(\Omega,\mathcal{F},(\mathcal{F}_t)_{t\geq0},\mathbb{P})$, a Banach space $(U,\|\cdot\|_U)$,  $q \in [2,\infty]$, and a stopping time $0 <\tau \leq T$, for some $T>0$. We  denote by  $\mathcal{X}$ the set of $U$-valued, $\left(\mathcal{F}_t\right)_{t \in[0, T]}$-adapted and continuous processes $X =\left\{X_{s}\right\}_{s \in[0, \tau]}$. We define 
	$$S_{\mathcal{F}}^{q}([0, \tau] ; U) = \left\{X \in \mathcal{X} \mid \Big\|\|X\|_{\tau,U}\Big\|_{q} < \infty\right\}.$$
		
		Also, let $\mathcal{Y}$ denote the set of  $U$-valued predictable processes $Y=\left\{Y_{s}\right\}_{s \in[0, \tau]}$. Then, for $q \in [2,\infty)$ we define
		$$L_{\mathcal{F}}^{q}([0,\tau];U) = \left\{Y \in \mathcal{Y}\mid\Big\|\|Y\|_{L^qU}\Big\|_{q} < \infty\right\}.$$
		
			These spaces are Banach spaces, with the norms $\Big\|\|\cdot\|_{\tau,U}\Big\|_q$ and $\Big\|\|\cdot\|_{L^qU}\Big\|_q$, respectively.

Let $f$ and $g$ be positive probability density functions on \(\mathbb{T}^d\). The relative entropy (or Kullback–Leibler divergence) of $f$ with respect to $g$ is defined as 
\begin{align*}
    \mathcal{H} \left(f|g\right)\doteq \int_{\mathbb{T}^d}f(x)\ln{\left(\frac{f(x)}{g(x)}\right)} \, dx.
\end{align*}

Also the Fisher information of $f$ with respect to $g$ is given by
\begin{align*}
    \mathcal{I}\left(f|g\right) \doteq \int_{\mathbb{T}^d}f(x)\left|\nabla \ln{\frac{f(x)}{g(x)}}\right|^2 \, dx = \int_{\mathbb{T}^d}\frac{(g(x))^2}{f(x)}\left|\nabla \left(\frac{f(x)}{g(x)}\right)\right|^2 \, dx.
\end{align*}
\subsection{Assumptions}
			
\begin{enumerate}
\item[$(\mathbf{A}^V)$] Let $V^N:\mathbb{T}^d\to \mathbb{R}$  the periodization of the scaling $V_0^N:\mathbb{R}^d \to \mathbb{R}$ given by $V_0^N(y) = N^{\beta}V_0(N^{\frac{\beta}{d}}y)$, $\beta \in (0,1)$, with $V_0:\mathbb{R}^d \to \mathbb{R}$ and
\begin{align}
    V_0(y)  \doteq 
     \begin{cases}
       \frac{\Gamma(d/2)}{2\pi^{d/2}}\exp{(-|y|^2)} &\quad |y|\leq 1/2\\
       \frac{\Gamma(d/2)}{2\pi^{d/2}}\exp{(1/4)}\exp{(-|y|)} &\quad |y|> 1/2.\\
      \end{cases}\label{def v0}
\end{align}


\item[$(\mathbf{A}^K)$]The kernel \( K \) is chosen so that \( \|K\|_1 < \infty \) and there exists \( q \ge 2 \) with \( q > d \), for which the following holds for every \( f \in L^q \left(\mathbb{T}^d\right) \): \[ \|K \ast f\|_{\infty} \le C_K\|f\|_{q}. \]

\item[$(\mathbf{A}^{K'})$]The kernel \( K \) is a Radon measure so that \( \|K\|_{BV} < \infty \).

\item[$(\mathbf{A}^{\nabla\cdot K})$] 
It holds $\nabla\cdot K =0$ and there exists $K_0 \in C^{\gamma}$, $\gamma \in \left(0,\frac{1}{2}\right)$ such that 
\begin{align*}
    K= \nabla\cdot K_0. \, 
\end{align*}
For examples of singular kernels that satisfy Assumptions $(\mathbf{A}^K)$, $(\mathbf{A}^{K'})$ and $(\mathbf{A}^{\nabla\cdot K})$, see Section \ref{example}.

\item[$(\mathbf{A}^{\sigma})$]
	The  coefficient  $\sigma: [0,T] \to \mathbb{R}^{d \times d}$ is measurable and bounded.
 \item[$(\mathbf{A}^{\rho_0})$]
	The initial condition $\rho_0$ is taken such that $\rho_0 \in \left(\lambda^{-1}, \lambda\right)$, for some $\lambda >1$, and
 \begin{align*}
       \rho_0 \in L^1 \cap H^{3}_q\left(\mathbb{T}^d\right).
 \end{align*}

	\end{enumerate}

  \begin{rem}     
We observe  that,
\begin{align}
    \left|\nabla V_0\right| \leq C_{d} V_0, \label{mollifier inequality}
\end{align}
and definition of periodization, for $x \in \mathbb{T}^d$: 
   \begin{align}
       V^N(x) &\doteq \sum_{k \in \mathbb{Z}^d} V_0^N(x - k). \label{def Vn period}
   \end{align}
   This series, along with the sequence of partial sums of its derivatives, converges uniformly on the torus, by virtue of Theorem 8.32 in \cite{folland2013real}. Then 
for a fixed  $N \in \mathbb{N}$ 
         \begin{align}   
   \int_{\mathbb{T}^d} V^N(x) \,dx  &= \int_{\mathbb{T}^d}\sum_{k \in \mathbb{Z}^d} V_0^N(x - k) dx \nonumber\\
   &=\sum_{k \in \mathbb{Z}^d} \int_{\mathbb{T}^d+k}V_0^N(y)\, dy =  \int_{\mathbb{R}^d}V_0^N(y)\, dy = 1. \label{vndensity}
   \end{align}   

Furthermore, since  $$\exp{(-|y|^2)}, \exp{(-|y|)} < {(2d)!}|y|^{-2d}$$
we have for $N \in \mathbb{N}$
\begin{align}
     V_0\left(N^{\frac{\beta}{d}}(x-k)\right) & \overset{(\ref{def v0})}{\leq} \frac{C_d}{\left|N^{\frac{\beta}{d}}\left(x - k \right)\right|^{2d}}\nonumber\\
       & \overset{}{=} \frac{C_d}{\left|N^{\frac{\beta}{d}}\right|^{2d}\left|\left(x -k \right)\right|^{2d}}, \, \, k \in \mathbb{Z}^d. \label{decay of mollifier}
   \end{align}   
  \end{rem}
 
  \begin{rem}
    We have chosen this specific mollifier $V_0$  in order to derive the inequality (\ref{weldef}) below, which shows that the relative entropy functional between the regularized empirical measure and the solution of the Fokker-Planck equation is well-defined.
      We  can take   the mollifier $V_0$  such that satisfies (\ref{mollifier inequality})-(\ref{decay of mollifier}).
  \end{rem}
\subsection{Statement of the main results.}

The next two results (whose proofs are provided in Appendix B) ensure that the limiting equation is well-posed and possesses the necessary regularity for the proofs of the main results.
\begin{thm}\label{Teo_Krylov} Assume $(\mathbf{A}^K)$,  $(\mathbf{A}^{\sigma})$ and $\nabla\cdot K = 0$.
Let $ \rho_0 \in \mathcal{P}\left(\mathbb{T}^d\right) \cap H^{1-\frac{2}{q}}_q(\mathbb{T}^d)$, with  $ \rho_0 \in \left(\lambda^{-1},  \lambda \right)$, for some $\lambda > 1$. There exists a time $T>0$ depending on $\norm{\rho_0}_{q},\lambda, q, C_K$ and $d$ such that the SPDE (\ref{SPDE_Ito}) admits a unique solution $\rho \in\left(\lambda^{-1}, \lambda \right) $, $\mathbb{P}$-a.s., and
\begin{align*}
 \rho \in L_{\mathcal{F}^B}^q\left([0, T]; H_q^{1}\left(\mathbb{T}^d\right)\right) \cap S_{\mathcal{F}^B}^{\infty}\left([0, T] ; L^1 \cap L^q\left(\mathbb{T}^d\right)\right).
\end{align*}
\end{thm}

\vspace{.5cm}
	\begin{cor}\label{coro Teo_Krylov} Under the conditions stated in Theorem \ref{Teo_Krylov}, together with Assumption $(\mathbf{A}^{\rho_0})$ and $K = \nabla\cdot K_0$, for some $K_0 \in L^{\infty}$,

    \vspace{.1cm}
    \begin{align}
         \rho \in L_{\mathcal{F}^B}^{q}\left([0,T];H^{4}_{q}\left(\mathbb{T}^d\right)\right) \cap S_{\mathcal{F}^B}^{q}\left([0,T];H^{3}_{q}\left(\mathbb{T}^d\right)\right). \nonumber
    \end{align}
\end{cor}

\vspace{.5cm}

 \begin{thm} \label{first main}
Assume $(\mathbf{A}^V)$, $(\mathbf{A}^K)$, $(\mathbf{A}^{\nabla\cdot K})$,   $(\mathbf{A}^{\sigma})$, and $(\mathbf{A}^{\rho_0})$, let $T_{max}$ be the maximal existence time for (\ref{SPDE_Ito}) and fix $T \in (0,T_{max})$. In addition, let the dynamics of the particle system be given by (\ref{particles})   and  
\begin{align*}
\lim_{N \to \infty} N^{\theta}\mathcal{H}(\rho_0^N|\rho_0)  =0, \, \, \, \mathbb{P}-a.s.
\end{align*}
where
\begin{align*}
\theta = \min \left(\beta(1-2\gamma);\frac{\beta}{d}\gamma^2;\frac{1}{2}-\beta\Big(1 + \frac{1 }{d}\Big) \right) - \delta 
\end{align*}
with  $\delta > 0$, such that $\theta >0$, $d \geq 1$ and $\beta \in \left(0,\frac{1}{2\big[1 + \frac{1}{d}\big]}\right) $.
Then
\begin{align*}
        \lim_{N \to \infty} N^{\theta}  \sup_{t \in [0,T]}\mathcal{H}(\rho_t^N|\rho_t) = 0, \, \, \mathbb{P}-a.s.
    \end{align*}
\noindent and $\rho$ is the unique solution of SPDE \eqref{SPDE_Ito} with initial condition $\rho_0$.
\end{thm}

			In view of the previous result, we obtain a rate of convergence for the genuine empirical measure, which can be interpreted as a propagation of chaos for the marginals of the empirical measure of the particle system. Following \cite[Section 8.3]{BogachevII}, let us introduce the Kantorovich-Rubinstein metric which reads, for any two probability measures $\mu$ and $\nu$ on $\mathbb{T}^d$,
			\begin{equation}\label{eq:defWasserstein}
				\|\mu - \nu \|_{0} = \sup \left\{ \int_{\mathbb{T}^d} \phi \, d(\mu-\nu) \, ; ~ \phi \text{ Lipschitz  with } \|\phi\|_{L^\infty}\leq 1 \text{ and } \|\phi\|_{\text{Lip}} \leq 1 \right\} .
			\end{equation}

			\begin{cor} \label{cor:rateEmpMeas}
				Let the same assumptions as in Theorem~\ref{first main}. Then   
				
				\[
			 \lim_{N \to \infty} N^{\theta-\delta}  \sup_{t\in[0,T]}  \left\|S_{t}^N - \rho_{t} \right\|_{0}^2 =0, \, \, \, \mathbb{P}-a.s.     
				\]
			\end{cor}
			
	\begin{proof} 
    Let $t\in (0,T_{max})$. We first observe that 
	there exists $C>0$ such that for any Lipschitz continuous function $\phi$ on $\mathbb{T}^d$, one has
				\begin{align}\label{eq:rateunmun}
					\left|\left\langle \rho^N_{t},\phi\right\rangle - \left\langle S^N_{t},\phi\right\rangle \right| \leq \frac{C \|\phi\|_{\text{Lip}}}{N^{\frac{\beta}{d}}}, \quad \, \, \, \mathbb{P}-a.s.
				\end{align}
				Indeed, 
\begin{align}
\left|\left\langle S_{t}^{N}, \phi\right\rangle - \left\langle \rho^N_t, \phi \right\rangle \right| &=\left| \left\langle S_{t}^{N}, \left(\phi-\phi\ast V^{N}\right)\right\rangle \right|
\nonumber\\
&\leq \left\langle S_{t}^{N}, \int_{\mathbb{R}^{d}} V_0(y)~ \left|\phi(\cdot)-  \phi\left(\cdot- \frac{y}{N^{\beta}}\right) \right|   dy \right\rangle \nonumber\\
&\leq \frac{C \|\phi\|_{\text{Lip}}}{N^{\frac{\beta}{d}}}. \nonumber
				\end{align}   
Now,  by (\ref{eq:defWasserstein}) and Lemma \ref{L1entropyfisherinformation inequality} in Appendix C,   

\begin{equation}\label{L1}
\left\|\rho^{N}_{t} - \rho_{t} \right\|_{0}^2\leq 
\left\|\rho^{N}_{t} - \rho_{t} \right\|_{L^{1}}^2
\leq C \mathcal{H}(\rho_t^N|\rho_t).
\end{equation}
By triangular inequality we have
    
\begin{align*}
    \sup_{t\in[0,T]}  \left\|S_{t}^N - \rho_{t} \right\|_{0} 
\leq  \sup_{t\in[0,T]}  \left\|S_{t}^N - \rho^{N}_{t} \right\|_{0} + \ 
 \sup_{t\in[0,T]}  \left\|\rho^{N}_{t} - \rho_{t} \right\|_{0}. 
\end{align*}

Thus applying (\ref{eq:rateunmun})   to the first term on the right-hand side of the above inequality, and Theorem \ref{first main} and \eqref{L1}   to the second term, we obtain Corollary \ref{cor:rateEmpMeas}.
    	\end{proof}

 \begin{thm} \label{second main}
Assume $(\mathbf{A}^V)$,$(\mathbf{A}^{K})$ or $(\mathbf{A}^{K'})$, $(\mathbf{A}^{\nabla\cdot K})$ and $(\mathbf{A}^{\rho_0})$  with $\sigma =0$, let $T_{max}$ be the maximal existence time for (\ref{SPDE_Ito}) and fix $T \in (0,T_{max})$. In addition, let the dynamics of the particle system be given by (\ref{particles})   and  for any $m \ge 1$,		\begin{align*}
\sup_{N \in \mathbb{N}} \Big \| \mathcal{H}(\rho_0^N|\rho_0) \Big \|_{{L^m(\Omega)}} < \infty
\end{align*}
with $ d \geq 1$ and  $\beta \in \left(0,\frac{1}{2\big[1 + \frac{1}{d}\big]}\right)$.
Then
\begin{align*}
\left \| \sup_{t\in [0,T]}\mathcal{H}(\rho_t^N|\rho_t) \right \|_{L^m(\Omega)}  &\lesssim  \Big \| \mathcal{H}(\rho_0^N|\rho_0) \Big \|_{L^m(\Omega)} +  N^{-\theta} \, \nonumber 
\end{align*}
where
\begin{align*}
\theta = \min \left(\beta(1-2\gamma);\frac{\beta}{d}\gamma^2;\frac{1}{2}-\beta\Big(1 + \frac{1 }{d}\Big) \right)
\end{align*}
\noindent and $\rho$ is the unique solution of PDE \eqref{SPDE_Ito} with initial condition $\rho_0$.
\end{thm}

 \begin{thm} \label{third main}
Assume $(\mathbf{A}^V)$,  $(\mathbf{A}^{\rho_0})$  and $K \in C^{\gamma}$, $\gamma \in \left(0,\frac{1}{2}\right)$. Let $T_{max}$ be the maximal existence time for (\ref{SPDE_Ito}) and fix $T \in (0,T_{max})$. In addition, let the dynamics of the particle system be given by (\ref{particles})   and  for any $m \ge 1$,		\begin{align*}
\sup_{N \in \mathbb{N}} \Big \| \mathcal{H}(\rho_0^N|\rho_0) \Big \|_{{L^m(\Omega)}} < \infty
\end{align*}
with $ d \geq 1$ and  $\beta \in \left(0,\frac{1}{2\big[1 + \frac{1}{d}\big]}\right) $.
Then
\begin{align*}
\left \| \sup_{t\in [0,T]}\mathcal{H}(\rho_t^N|\rho_t) \right \|_{L^m(\Omega)}  &\lesssim \Big \| \mathcal{H}(\rho_0^N|\rho_0) \Big \|_{L^m(\Omega)} +  N^{-\theta} \, \nonumber 
\end{align*}
where
\begin{align*}
\theta = \min \left(\beta(1-2\gamma);\frac{\beta}{d}\gamma^2;\frac{1}{2}-\beta\Big(1 + \frac{1 }{d}\Big) \right)
\end{align*}
\noindent and $\rho$ is the unique solution of SPDE \eqref{SPDE_Ito} with initial condition $\rho_0$.
\end{thm}

\begin{rem} We observe that, compared with  previous works on moderate particle systems 
without common noise,  we lose the order of convergence. This is because the estimation of the martingale term in our work is of order $\frac{1}{2}-\beta\Big(1 + \frac{1 }{d}\Big)$, while in  \cite{Josue}, \cite{Pisa} is the order  $\frac{1}{2}- \frac{\beta}{2}$, see Theorem 1.3 in  \cite{Pisa} and Theorem 1 in \cite{Josue}. For the same reason, the range of the parameter $\beta$  is smaller in our work. 
\end{rem}

\subsection{Applications} \label{example}
\textbf{Sub-Coulomb kernels.} The following example was explored in \cite{Carrillo2014} and \cite{quasibiot}.
Let
\begin{align}
    K(x_1,x_2) &\doteq   c\frac{(x_1,x_2)^{\perp}}{|(x_1,x_2)|^{1+\alpha}} + c \sum_{(k_1,k_2) \neq (0,0)} \frac{(x_1 -k_1,x_2 -k_2)^{\perp}}{|(x_1-k_1,x_2-k_2)|^{1+\alpha}}, \, \, \, (x_1,x_2) \in \mathbb{T}^2,\nonumber
\end{align}
with $\alpha \in (0,1)$ and $c$ is positive or negative constant. Now, we define
 \begin{align*}
 K_0(x_1,x_2)\doteq   \begin{bmatrix}
0 & -\frac{|(x_1,x_2)|^{1-\alpha}}{1-\alpha} + \psi_1  \\
\frac{|(x_1,x_2)|^{1-\alpha}}{1-\alpha} + \psi_2 & 0  \\
\end{bmatrix}, \, \, \, (x_1,x_2) \in \mathbb{T}^2
\end{align*}
\pagebreak

\noindent where $(\psi_1,\psi_2)$ are smooth corrections of periodization. So $K_0 \in C^{1-\alpha}$, 
$K(x_1,x_2) =  c \nabla \cdot K_0(x_1,x_2)$ and 
$\nabla\cdot K =0$. Then $K$ fulfills the assumption $(\mathbf{A}^{\nabla\cdot K})$. 

Also, given $f \in L^q$, $q > 2$, since $H^1_q \hookrightarrow C^{1- \frac{2}{q}}$
\begin{align} \label{quasi biot satisfies A grad K}
    \left\|K\ast f\right\|_{\infty} & \leq  \left\|K\ast f\right\|_{1-\frac{2}{q}}\overset{}{\leq} \left\|K\ast f\right\|_{1,q}\doteq \left\|K\ast f\right\|_{q} + \left\|\nabla (K\ast f\right)\|_{q}.
\end{align}
Now since 
\begin{align*}
    \left|K(x_1,x_2)\right|, \left|(x_1,x_2)\right| \left|\nabla K(x_1,x_2)\right| \leq \frac{C}{\left|(x_1,x_2)\right|^{\alpha}}, \, \, \, (x_1,x_2) \in \mathbb{T}^2,
\end{align*}
implies that $K, \nabla K \in L^{1}\left(\mathbb{T}^d\right)$, by convolution inequality,
\begin{align*}
    \left\|K\ast f\right\|_{q} + \left\|\nabla (K\ast f\right)\|_{q} \leq C\|f\|_q.
\end{align*}
From the estimate in $(\ref{quasi biot satisfies A grad K})$,  $K$ satisfies the assumption $(\mathbf{A}^{K})$. It follows that this singular kernel is covered by Theorems \ref{Teo_Krylov}, \ref{first main} and \ref{second main}, as well as their Corollaries.

\vspace{.5cm}
\label{example 1}
\noindent \textbf{Hausdorff measure kernel.} As noted in \cite{Gui} and \cite{Jabin_2018}, we can handle  singular kernels like measures. 

Let $C:\left[-\frac{1}{2},\frac{1}{2}\right] \to \left[-\frac{1}{2},\frac{1}{2}\right]$ be the standard Cantor function and consider

 \begin{align*}
 K_0(x_1,x_2)\doteq   \begin{bmatrix}
- x_2C(x_1) &0  \\
 0 & x_1C(x_2)  \\
\end{bmatrix}, \, \, \, (x_1,x_2)  \in \mathbb{T}^2.
\end{align*}

Since the distributional derivative of the Cantor function is the Hausdorff measure $\mu$ concentrated on the Cantor set, it is a finite measure and by design 

\begin{align*}
    K(x_1,x_2) &\doteq \nabla \cdot K_0(x_1,x_2) = (x_1,x_2)^{\perp} \mu, \, \, \, (x_1,x_2) \in \mathbb{T}^2.
   \end{align*}
   
   Thus, $K$ fulfills the assumption $(\mathbf{A}^{\nabla\cdot K})$, taking into account that the Cantor function is a Holder continuous function, with exponent $\frac{\ln{2}}{\ln{3}}$. Also, since $\mu$ is finite, it is of bounded variation, which means $\|K\|_{BV} < \infty$. Finally, by considering the solution to equation (\ref{SPDE_Ito}) with \(\sigma = 0\), as provided in \cite{Gui}, Theorem \ref{second main} becomes applicable to this singular kernel.
\pagebreak

\subsection{Definition of solution}

\begin{defn} A family of random functions $\left\{\rho_t(\omega): t \ge 0, \omega \in \Omega\right\}$ lying in $L_{\mathcal{F}^B}^q\left([0, T]; H_q^{1}\left(\mathbb{T}^d\right)\right) \cap S_{\mathcal{F}^B}^{\infty}\left([0, T] ; L^1 \cap L^q\left(\mathbb{T}^d\right)\right)$ is a solution to (\ref{SPDE_Ito}) if $\rho_t$ satisfies for all $\phi \in C^2\left(\mathbb{T}^d\right)$,

\begin{align}
\left\langle \rho_t, \phi\right\rangle &= \left\langle \rho_0, \phi\right\rangle+ \int_0^t\left\langle \rho_s K \ast \rho_s, \nabla \phi \right\rangle ds \nonumber \\
& \quad +\frac{1}{2} \int_0^t\Big\langle \rho_s, \sum_{i,j=1}^d \partial_{ij}\phi \,   \sum_{k=1}^d (\delta^{ik} \, \delta^{jk}+ \sigma_s^{ik} \, \sigma_s^{jk}) \Big\rangle \, ds \nonumber \\
&\quad +\int_0^t\Big\langle \rho_s, \sum_{i=1}^d \partial_i \phi\,  \sum_{k=1}^{d} \sigma_s^{ik} \Big\rangle \, d B_s^k , \, \, \, \mathbb{P}-a.s.  \label{weak_sol}
\end{align}
\end{defn}


\vspace{1 cm}

\section{Proofs of main results}

 This section is devoted to the proofs of our main results. In the subsection \ref{timeevolution}, we  derive an  evolution equation for the relative entropy functional of the regularized empirical measure with respect to the solution of the Fokker-Planck equation (\ref{SPDE_Ito}).  In the subsection \ref{fissher}, we identify   the dissipation terms associated with the Fisher information and we obtain some  estimates for the quadratic variation terms.  In the subsections \ref{nonlinearnholder} and \ref{nonlinearholder}, we address the nonlinear terms appearing in the evolution equation for the relative entropy functional, derived in the first step.
Finally, in subsection \ref{end of the proof}, we deal with the martingale terms in our computations.  Then, combining this with the previous estimates, Gronwall’s Lemma allows us to close the argument.

\subsection{Time evolution of the relative entropy} 
\label{timeevolution}

In this subsection we derive an evolution equation for the relative entropy functional, relating the regularized empirical measure and the solution of the Fokker-Planck equation (\ref{SPDE_Ito}), using the Itô's formula. Additionally, we  check that the relative entropy functional is well-defined (see  (\ref{weldef})).

Applying the It{\^o}'s formula with $V^N (x - \cdot)$ for each $i \in \{
1, \ldots, N \}$ in $(\ref{particles})$,  denoting $X_s^{i,N} = X_s^i$ and $W_s^{i,N}=W_s^i$, we have

\begin{align*}
     V^N (x - X_t^{i}) & = \, V^N (x - X_0^{i}) \\ 
     &- \sum_{j=1}^d\int_0^t\partial_jV^N (x - X_s^{i})       (K\ast \rho_s^N)_j (X_s^{i}) \, ds\\
     & + \frac{1}{2}\sum_{j,k=1}^d \int_0^t\partial_{j,k}V^N (x - X_s^{i}) 
     (\sigma \sigma^\top)_s^{j,k} \, ds\\
      & - \sum_{j,k=1}^d \int_0^t\partial_{j}V^N (x - X_s^{i})  
     \sigma_s^{j,k} \, d  B_s^k\\
     & \,  \, - \sum_{j=1}^d\int_0^t\partial_jV^N (x - X_s^{i})  \, d W_s^{i} \\
     &+ \frac{1}{2} \sum_{j=1}^d \int_0^t \partial_{jj} V^N (x - X_s^{i}) \, d s,
   \end{align*}
   with $(\sigma \sigma^\top)^{j,k}_s \doteq  \sum_{l = 1}^{d} \sigma^{jl}_s\sigma^{kl}_s$. Then since that $\rho^N = V^N \ast S^N$ we obtain 
   \begin{align}
    \rho_t^N(x) & = \, \rho_0^N(x) \nonumber \\
    &- \sum_{j=1}^d\int_{0}^{t}\left\langle S_{s}^{N}, \partial_jV^{N}(x-\cdot)\left( K\ast \rho_{s}^{N}\right)_j(\cdot)\right\rangle \, ds \nonumber\\
     & + \frac{1}{2}\sum_{j,k=1}^d\int_{0}^{t} \partial_{j,k} \rho_s^N(x) (\sigma \sigma^\top)_s^{j,k} \, ds \nonumber\\
      & -\sum_{j,k=1}^d \int_0^t \partial_{j}\rho_s^N(x)
     \sigma_s^{j,k} \, d  B_s^k\nonumber\\
     & \,  \, - \frac{1}{N}\sum_{i=1}^N\sum_{j=1}^d\int_0^t\partial_jV^N (x - X_s^{i})  \, d  W_s^{i} \nonumber\\
     & \,  \, + \frac{1}{2}\sum_{j=1}^d  \int_0^t \partial_{jj}\rho_s^N(x) \, ds . \label{Ndensity}
   \end{align}

In addition,  regarding the solution of the Fokker-Planck equation (\ref{SPDE_Ito}), since \( 3 > 2 + \frac{d}{q} \) for \( q > d \), Corollary \ref{coro Teo_Krylov} combined with Sobolev embedding, we find that \( \rho_t \in C^2(\mathbb{T}^d) \), $\mathbb{P}$-a.s., for all \( t \in [0, T] \). Then it is a classical  semimartingale and  verifies 
\begin{align}
    \rho_t(x) & = \, \rho_0(x) \nonumber\\
    &- \sum_{j=1}^d\int_{0}^{t} \partial_j\left[\rho_s(x) (K\ast \rho_{s})_j(x) \right] d s\nonumber\\
     & + \frac{1}{2}\sum_{j,k=1}^d\int_{0}^{t} \partial_{j,k}\rho_s(x) (\sigma \sigma^\top)_s^{j,k} d s\nonumber\\
      & - \sum_{j,k=1}^d \int_0^t\partial_{j}\rho_s(x)\sigma_s^{j,k}
        \, d  B_s^k\nonumber\\
         & \,  \, + \frac{1}{2}\sum_{j=1}^d  \int_0^t \partial_{jj} \rho_s(x) \, d s . \label{limit eq}
   \end{align}
   
Now by $(\ref{def Vn period})$, $ \rho_s^N(x)> 0$ 
  for a fixed $s \in [0,T]$  and $N \in \mathbb{N}$.
   So, due to the fact that $\rho_s^N$ is a positive and smooth function on $\mathbb{T}^d$, 
   by applying the Itô's formula 
   to (\ref{Ndensity}), we have
   
\begin{align} \label{itô to spde particle}
    \rho_t^N\ln{(\rho_t^N)} &=  \rho_0^N\ln(\rho_0^N)\nonumber\\
    &- \int_0^t\left(1+\ln(\rho_s^N) \right)\left\langle S_s^N, \nabla V^N(x-\cdot)K\ast \rho_s^N(\cdot)\right\rangle \,ds \nonumber \\
    & + \frac{1}{2}\int_{0}^{t}\left(1+\ln(\rho_s^N) \right) (\sigma \sigma^\top)_s D^2\rho_s^N \, ds \nonumber\\
    &+ \frac{1}{2}\int_0^t\left(1+\ln(\rho_s^N) \right)\Delta \rho_s^N \,ds \nonumber \\
    & - \int_0^t\left(1+\ln(\rho_s^N) \right)\sigma^{\top}_s \nabla \rho_s^N \,d B_s\nonumber\\
    &- \frac{1}{N}\sum_{i=1}^N \int_0^t\left(1+\ln(\rho_s^N) \right)\nabla V^N(x - X_s^i)\,dW^i_s\nonumber\\
     &+  \int_0^{t}\frac{1}{\rho_s^N}\left|\sigma^{\top}_s  \nabla \rho_s^N\right|^2 \, ds   \nonumber\\
    &+ \frac{1}{2N^2}\sum_{i=1}^N \int_0^t \frac{1}{\rho_s^N}\left|\nabla V^N(x - X_s^i)\right|^2 \,ds.
\end{align}
\pagebreak

  We have used above the following expressions for the quadratic variation terms. First, since our diffusion parameter $\sigma$ does not depends on the spacial variable we derive

\begin{align*}
&\left\langle -\sum_{j,k=1}^d \int_0^{\cdot}\left\langle S_s^N, \sigma_s^{j,k}\partial_{j}V^N (x - \cdot) 
      \right \rangle \, d  B_s^k\right\rangle_t\\ &=\left\langle \sum_{k=1}^d \int_0^{\cdot}\left\langle S_s^N, \sum_{j=1}^d (\sigma^{\top}_{s})^{k,j}\partial_{j}V^N (x - \cdot) 
     \right \rangle \, d  B_s^k\right\rangle_t \\
     &=\left\langle \int_0^{\cdot}\left\langle S_s^N, \sigma^{\top}_s  \nabla V^N (x - \cdot) 
     \right \rangle \, d  B_s\right\rangle_t \\
     &=\int_0^{t}\left|\left\langle S_s^N, \sigma^{\top}_s  \nabla V^N (x - \cdot) 
     \right \rangle \right|^2 \, ds \\
      &=\int_0^{t}\left|\sigma^{\top}_s \nabla\rho_s^N \right|^2 \, ds. 
   \end{align*}

\vspace{.7 cm}
   
   Additionally, we deduce 
   \begin{align*}
\left\langle - \frac{1}{N} \sum_{i=1}^N\sum_{j=1}^d\int_0^{\cdot}\partial_jV^N (x - X_s^{i}) \, d  W_s^{i}\right\rangle_t &=  \left\langle \frac{1}{N} \sum_{i=1}^N\int_0^{\cdot}\nabla V^N (x - X_s^{i}) \,  d  W_s^{i}\right\rangle_t
\\
&=  \frac{1}{N^2} \sum_{i=1}^N\left\langle \int_0^{\cdot}\nabla V^N (x - X_s^{i}) \, d  W_s^{i}\right\rangle_t
\\
 &= \frac{1}{N^2} \sum_{i=1}^N\int_0^t | \nabla V^N (x - X_s^{i}) |^2 \, ds. 
   \end{align*}

\vspace{1 cm}

We recall by Theorem \ref{Teo_Krylov}, Corollary \ref{coro Teo_Krylov} and Sobolev embedding, $\rho_t \in C^2\left(\mathbb{T}^d\right)$ $\mathbb{P}$-a.s., for all $t\in [0,T]$.
Then, taking into account (\ref{Ndensity}) and (\ref{limit eq}), since  $W^i$ and $B$ are independent Brownian motions, applying the Ito's formula 
to $(\rho_t^N, \rho_t)$, we get

\begin{align} \label{ito for product}
    \rho_t^N \ln(\rho_t) &= \rho_0^N\ln(\rho_0) \nonumber \\
    &- \int_0^t \ln(\rho_s)\left \langle S_s^N,\nabla V^N(x-\cdot)K\ast \rho_s^N(\cdot) \right\rangle \,ds \nonumber \\
     &- \int_0^t\frac{\rho_s^N}{\rho_s} \nabla \cdot (\rho_s (K \ast \rho_s)) \,ds \nonumber \\
    & + \frac{1}{2} \int_0^t\ln(\rho_s)\Delta \rho_s^N \,ds \nonumber\\
    & + \frac{1}{2} \int_0^t \frac{\rho_s^N}{\rho_s}\Delta \rho_s\,ds \nonumber\\
     & + \frac{1}{2}\int_{0}^{t} \ln(\rho_s) (\sigma \sigma^\top)_s D^2\rho_s^N  \,ds \nonumber\\
    & + \frac{1}{2}\int_{0}^{t} \frac{\rho_s^N}{\rho_s} D^2\rho_s(x) (\sigma \sigma^\top)_s \, d s\nonumber\\
    & -\int_0^t\ln(\rho_s)\sigma^{\top}_s \nabla \rho_s^N \, d  B_s\nonumber\\
    &- \frac{1}{N}\sum_{i=1}^N \int_0^t \ln(\rho_s) \nabla V^N(x - X_s^i) \,dW_s^i\nonumber\\
     & -\int_0^t\frac{\rho_s^N}{\rho_s} \sigma^{\top}_s\nabla \rho_s \, d  B_s\nonumber\\
     &+\int_0^t \frac{1}{\rho_s} \sigma^{\top}_s  \nabla \rho_s^N \sigma^{\top}_s\nabla \rho_s \, ds \nonumber\\
     &- \frac{1}{2}\int_0^t \frac{\rho_s^N}{(\rho_s)^2} \left|\sigma^{\top}_s \nabla \rho_s\right|^2 \, ds.
   \end{align}

\vspace{1 cm}

 Now we verify that the relative entropy functional is well-defined, namely  $|\mathcal{H}\left(\rho_t^N|\rho_t\right)| < \infty$, $\mathbb{P}$-a.s.,  for all $N \in \mathbb{N}$ and $t \in [0,T]$.

First we observe, for $\alpha \in \{1,2\}$ 
\begin{align*}
    |x - X_s^i| \leq 1 \, \implies -N^\frac{\alpha\beta}{d} \leq -\left| N^{\frac{\beta}{d}}(x - X_s^i)\right|^{\alpha},
\end{align*}
which implies
\begin{align}
 N^{\beta}\frac{\Gamma(d/2)}{2\pi^{d/2}}\exp\left(-N^{\frac{2\beta}{d}}\right) &\leq
    N^{\beta}\frac{\Gamma(d/2)}{2\pi^{d/2}}\exp\left(-N^{\frac{\alpha\beta}{d}}\right)\nonumber\\ &\leq \frac{1}{N} \sum_{i=1}^N N^{\beta}\frac{\Gamma(d/2)}{2\pi^{d/2}}\exp\left(-\left| N^{\frac{\beta}{d}}(x - X_s^i)\right|^{\alpha}\right)\nonumber\\
    &\leq \frac{1}{N} \sum_{i=1}^N N^{\beta}\frac{\Gamma(d/2)}{2\pi^{d/2}}\exp\left(-\left| N^{\frac{\beta}{d}}(x - X_s^i)\right|^{\alpha}\right)\nonumber\\
    &+\frac{1}{N} \sum_{i=1}^N \sum_{k\neq 0}N^{\beta}\frac{\Gamma(d/2)}{2\pi^{d/2}}\exp{(1/4)}\exp\left(-\left| N^{\frac{\beta}{d}}\left[(x - X_s^i)-k\right]\right|\right) \nonumber\\
    &\doteq A + B. \label{lower boound for ron}
    \end{align}
    Note that by $(\ref{def v0})$,  we have
    \begin{align}
        B=\frac{1}{N} \sum_{i=1}^N \sum_{k\neq 0}N^{\beta}V_0\left( N^{\frac{\beta}{d}}\left[(x - X_s^i)-k\right]\right), \label{B}
    \end{align}
    since $(x - X_s^i) \in \mathbb{T}^d$ implies for $k \neq 0$, $\left|(x - X_s^i)-k\right|\geq 1/2$, and then $\left|N^{\frac{\beta}{d}}\left[\left(x-X_s^i\right) - k \right]\right|\geq 1/2$.
    
    Concerning the term $A$, 
    by $(\ref{def v0})$,  if $\left| N^{\frac{\beta}{d}}(x - X_s^i)\right|> 1/2$ and $\alpha = 1$ in $(\ref{lower boound for ron})$, we obtain 
    \begin{align}
        A\leq\frac{1}{N} \sum_{i=1}^N N^{\beta}V_0\left( N^{\frac{\beta}{d}}\left[(x - X_s^i)\right]\right). \label{A1}
    \end{align}
    If $\left| N^{\frac{\beta}{d}}(x - X_s^i)\right|\leq 1/2$ and $\alpha = 2$, in $(\ref{lower boound for ron})$ we get
    \begin{align}
        A=\frac{1}{N} \sum_{i=1}^N N^{\beta}V_0\left( N^{\frac{\beta}{d}}\left[(x - X_s^i)\right]\right).\label{A11}
    \end{align}
 From  $(\ref{lower boound for ron})$,  $(\ref{B})$, $(\ref{A1})$, and  $(\ref{A11})$  we find 
\begin{align}
    N^{\beta}\frac{\Gamma(d/2)}{2\pi^{d/2}}\exp\left(-N^{\frac{2\beta}{d}}\right)&\leq \frac{2}{N} \sum_{i=1}^N V^N(x - X_s^i) = 2\rho_s^N(x). \label{lower regularized}
\end{align}
  Now, we observe that
 \begin{align}
       V^N(x - X_s^i) &\overset{(\ref{def Vn period})}{=} \sum_{k \in \mathbb{Z}^d} N^{\beta} V_0\left(N^{\frac{\beta}{d}}((x - X_s^i)-k)\right)\nonumber\\
       &= N^{\beta} V_0\left(N^{\frac{\beta}{d}}(x-X_s^i)\right) + \sum_{k \neq 0} N^{\beta} V_0\left(N^{\frac{\beta}{d}}((x - X_s^i)-k)\right)\nonumber\\
       &\overset{(\ref{def v0})+(\ref{decay of mollifier})}{\leq }N^{\beta}\exp{(1/4)} + N^{\beta}\sum_{k \neq0}\frac{C_d}{\left|N^{\frac{\beta}{d}}\right|^{2d}\left|\left((x - X_s^i) -k \right)\right|^{2d}} \nonumber\\
       &\leq 2 N^{\beta}\exp{(1/4)}C_d. \label{estimate vn}
   \end{align}
    
    By joining $(\ref{lower regularized})$ and $(\ref{estimate vn})$, we deduce
    
    \begin{align} 
    \frac{1}{2}N^{\beta}\frac{\Gamma(d/2)}{2\pi^{d/2}}\exp\left(-N^{\frac{2\beta}{d}}\right)&\leq \frac{1}{N} \sum_{i=1}^N V^N(x - X_s^i) \nonumber\\
    &=  \rho_s^N(x) \, \leq 2 N^{\beta}\exp{(1/4)}C_d.\label{bound for rhoN}
\end{align}

Thus, by $(\ref{bound for rhoN})$ we arrive at
\begin{align}
   -C_N \leq \ln{\left[ \frac{1}{2}N^{\beta}\frac{\Gamma(d/2)}{2\pi^{d/2}}\exp\left(-N^{\frac{2\beta}{d}}\right)\right]} \leq \ln{\left(\rho_s^N\right)} \leq \ln\left(2N^{\beta}\exp{(1/4)}C_d\right) \leq C_N, \label{entropy well defined}
\end{align}
with  $$C_N \doteq \left( \ln{(N^\beta)} -\ln{\left(\frac{\Gamma(d/2)}{2\pi^{d/2}}\right)} + N^{\frac{2\beta}{d}} + 2\ln(2) + 1/4 + C_d\right),$$
since that
$$\ln{\left(\frac{\Gamma(d/2)}{2\pi^{d/2}}\right)} \leq 0.$$
From  Theorem \ref{Teo_Krylov}, $\rho \in \left(\lambda^{-1},\lambda\right)$, $\lambda >1$, $\mathbb{P}$-a.s., then by (\ref{entropy well defined}) we obtain
\begin{align}
  \label{weldef}\left|\mathcal{H}\left(\rho_s^N|\rho_s\right)\right| &\leq \int_{\mathbb{T}^d}\left|\rho_s^N(x)\ln{\rho_s^N(x)}\right| \, dx + \int_{\mathbb{T}^d}\left|\rho_s^N(x)\ln{\rho_s(x)}\right| \, dx \nonumber\\
  &\overset{(\ref{estimate vn})}{\leq} 2N^{\beta}\exp{(1/4)}C_d\left(C_N + \ln{\lambda}\right) < \infty. 
\end{align}
Finally, from  (\ref{itô to spde particle}) and (\ref{ito for product}), we have the following identity
\begin{align}
    \mathcal{H}(\rho_t^N|\rho_t)- \mathcal{H}(\rho_0^N|\rho_0) 
    &= \int_{\mathbb{T}^d} \int_0^t \left[\ln(\rho_s) - \ln(\rho_s^N)\right] \left\langle S_s^N, \nabla V^N(x - \cdot)K\ast \rho_s^N(\cdot)\right\rangle \,ds \,dx \nonumber \\
    &+ \int_{\mathbb{T}^d} \int_0^t \left[-\left\langle S_s^N, \nabla V^N(x - \cdot) K \ast \rho_s^N(\cdot) \right\rangle + \frac{\rho_s^N}{\rho_s}\nabla\cdot(\rho_s (K\ast \rho_s)) \right] \,ds\,dx \nonumber \\
    &+ \frac{1}{2} \int_{\mathbb{T}^d} \int_0^t \left[\ln(\rho_s^N) - \ln(\rho_s)\right] (\sigma \sigma^\top)_sD^2 \rho_s^N \,ds\,dx\nonumber  \\
     &+ \frac{1}{2} \int_{\mathbb{T}^d} \int_0^t (\sigma \sigma^\top)_s D^2 \rho_s^N \,ds\,dx - \frac{1}{2} \int_{\mathbb{T}^d}\int_{0}^{t} \frac{\rho_s^N}{\rho_s} D^2\rho_s (\sigma \sigma^\top)_s \, d s \, dx\nonumber\\
    &+ \frac{1}{2} \int_{\mathbb{T}^d} \int_0^t \left[\ln(\rho_s^N) - \ln(\rho_s)\right] \Delta \rho_s^N \,ds\,dx \nonumber  \\
    &+ \frac{1}{2} \int_{\mathbb{T}^d} \int_0^t \left[-\frac{\rho_s^N}{\rho_s} \Delta \rho_s + \Delta \rho_s^N\right] \,ds \,dx \nonumber \\
    & +\int_{\mathbb{T}^d}\int_0^t\left[\ln(\rho_s) -\ln(\rho_s^N)\right]\sigma^{\top}_s \nabla \rho_s^N \, d  B_s \, dx\nonumber\\
       & +\int_{\mathbb{T}^d}\int_0^t\sigma^{\top}_s \nabla \rho_s^N \, d  B_s \, dx \nonumber\\
     &+ \frac{1}{N} \sum_{i=1}^N \int_{\mathbb{T}^d} \int_0^t \left[\ln(\rho_s) - \ln(\rho_s^N)\right]\nabla V^N(x - X_s^i)\,dW^i_s \,dx \nonumber \\
    &- \frac{1}{N} \sum_{i=1}^N \int_{\mathbb{T}^d}\int_0^t \nabla V^N(x - X_s^i)\,dW^i_s \,dx  \nonumber \\
     & + \int_{\mathbb{T}^d} \int_0^t\frac{\rho_s^N}{\rho_s} \sigma^{\top}_s\nabla \rho_s \, d  B_s \, dx  \nonumber\\
     &+ \frac{1}{2}\int_{\mathbb{T}^d}\int_0^{t}\frac{1}{\rho_s^N}\left|\sigma^{\top}_s \nabla \rho_s \right|^2 \, ds \, dx\nonumber\\
    &+ \frac{1}{2N^2}\sum_{i=1}^N \int_{\mathbb{T}^d} \int_0^t \frac{1}{\rho_s^N}|\nabla V^N(x - X_s^i)|^2 \,ds\,dx\nonumber \\
     &- \int_{\mathbb{T}^d}\int_0^t \frac{1}{\rho_s} \sigma^{\top}_s  \nabla \rho_s^N \sigma^{\top}_s\nabla \rho_s \, ds \, dx\nonumber\\
     &+ \frac{1}{2}\int_{\mathbb{T}^d}\int_0^t \frac{\rho_s^N}{(\rho_s)^2} \left|\sigma^{\top}_s \nabla \rho_s\right|^2 \, ds \, dx\nonumber\\ 
     &\doteq \sum_{l=1}^{15} R_t^l  \nonumber\\
     &= I_t + II_t + III_t + M_t^N + IV_t,\label{ito at entropy}
\end{align}
where 

\begin{align*}
    &I_t \doteq R_t^1 + R_t^2\\
    &II_t \doteq R_t^3 + R_t^4 \\
     &III_t \doteq R_t^5 + R_t^6 \\
      &M^N_t \doteq R_t^7 + R_t^8 + R_t^9 +R_t^{10} + R_t^{11} \\
       &IV_t \doteq R_t^{12} + R_t^{13} + R_t^{14} + R_t^{15}. \\
\end{align*}

\vspace{.7 cm}

\subsection{Estimates for $II_t$, $III_t$ and $IV_t$} \label{fissher}

 Now, we identify the dissipation that allows us to control the nonlinear terms and derive estimates for the quadratic variations arising from Itô’s formula, using the  property (\ref{mollifier inequality}) of the mollifier $V_0$.

 We begin by determining the dissipation terms related with the Fisher information. 

Recall that $(\sigma \sigma^\top)_s\doteq \sum_{l=1}^d\sigma_s^{jl}\sigma_s^{kl}$. By Fubini's theorem and periodic boundary conditions, we have 

\begin{align*}
    II_t &\doteq 
     \frac{1}{2} \sum_{j,k=1}^d \int_0^t\int_{\mathbb{T}^d}  \left[\ln(\rho_s^N) - \ln(\rho_s)\right]  \partial_{j,k}\rho_s^N (\sigma \sigma^\top)_s^{j,k} \, \,dx \, ds \nonumber  \\
     &+ \frac{1}{2} \int_0^t (\sigma \sigma^\top)_s \underbrace{\int_{\mathbb{T}^d}   D^2 \rho_s^N \,dx}_{=0}\,ds - \frac{1}{2}\sum_{j,k=1}^d \int_{0}^{t} \int_{\mathbb{T}^d} \frac{\rho_s^N}{\rho_s} \partial_{j,k}\rho_s (\sigma \sigma^\top)^{j,k}_s \, d x \, ds\\
      &= \frac{1}{2} \sum_{j,k=1}^d \int_0^t\int_{\mathbb{T}^d}  \left[\ln(\rho_s^N) - \ln(\rho_s)\right]\partial_{j,k}\rho_s^N \sum_{l=1}^d\sigma_s^{jl}\sigma_s^{kl} \,dx \, ds\nonumber  \\
     & - \frac{1}{2}\sum_{j,k=1}^d \int_{0}^{t} \int_{\mathbb{T}^d} \frac{\rho_s^N}{\rho_s} \partial_{j,k}\rho_s \sum_{l=1}^d\sigma_s^{jl}\sigma_s^{kl} \, d x \, ds.\\
\end{align*}

Now by integration by parts and Leibnz's rule, we get
\begin{align*}
    II_t &=  -\frac{1}{2} \sum_{j,k,l=1}^d \int_0^t\int_{\mathbb{T}^d}  \partial_j\left[\ln(\rho_s^N) - \ln(\rho_s)\right] \partial_{k}\rho_s^N \sigma_s^{jl}\sigma_s^{kl}\,dx \, ds\nonumber  \\
     & + \frac{1}{2}\sum_{j,k,l=1}^d \int_{0}^{t} \int_{\mathbb{T}^d}\partial_j\left( \frac{\rho_s^N}{\rho_s}\right) \partial_{k}\rho_s \sigma_s^{jl}\sigma_s^{kl} \, d x \, ds\\
     &=  -\frac{1}{2} \sum_{j,k,l=1}^d \int_0^t\int_{\mathbb{T}^d}  \partial_j\left[\ln(\rho_s^N) \right] \partial_{k}\rho_s^N \sigma_s^{jl}\sigma_s^{kl}\,dx \, ds \nonumber  \\
     &+\frac{1}{2} \sum_{j,k,l=1}^d \int_0^t\int_{\mathbb{T}^d}  \partial_j\left[ \ln(\rho_s)\right] \partial_{k}\rho_s^N \sigma_s^{jl}\sigma_s^{kl}\,dx \, ds \nonumber  \\
     & + \frac{1}{2}\sum_{j,k,l=1}^d \int_{0}^{t} \int_{\mathbb{T}^d}\left( \frac{\partial_j\rho_s^N\rho_s - \rho_s^N \partial_j \rho_s}{(\rho_s)^2}\right) \partial_{k}\rho_s \sigma_s^{jl}\sigma_s^{kl} \, dx \, ds.
\end{align*}

Thus, we find 
\begin{align*}
    II_t &= -\frac{1}{2} \sum_{j,k,l=1}^d \int_0^t\int_{\mathbb{T}^d}  \frac{\partial_j \rho_s^N}{\rho_s^N} \partial_{k}\rho_s^N \sigma_s^{jl}\sigma_s^{kl}\,dx ds \nonumber  \\
     &+\frac{1}{2} \sum_{j,k,l=1}^d \int_0^t\int_{\mathbb{T}^d}  \frac{\partial_j \rho_s}{\rho_s} \partial_{k}\rho_s^N \sigma_s^{jl}\sigma_s^{kl}\,dx ds\nonumber  \\
     & + \frac{1}{2}\sum_{j,k,l=1}^d \int_{0}^{t} \int_{\mathbb{T}^d}\left( \frac{\partial_j\rho_s^N }{\rho_s}\right) \partial_{k}\rho_s \sigma_s^{jl}\sigma_s^{kl} d x ds\\
     & - \frac{1}{2}\sum_{j,k,l=1}^d \int_{0}^{t} \int_{\mathbb{T}^d}\left( \frac{ \rho_s^N \partial_j \rho_s}{(\rho_s)^2}\right) \partial_{k}\rho_s \sigma_s^{jl}\sigma_s^{kl} d x ds\\    
     &= -\frac{1}{2}  \int_0^t\int_{\mathbb{T}^d} \frac{\left|\sigma^{\top}_s \nabla \rho_s^N\right|^2}{\rho_s^N} \, dx ds \nonumber  \\
     &+\int_0^t\int_{\mathbb{T}^d}  \frac{\left(\sigma^{\top}_s\nabla \rho_s^N\right)\left(\sigma^{\top}_s\nabla \rho_s\right)}{\rho_s}\,dx ds \nonumber  \\
     & - \frac{1}{2} \int_{0}^{t} \int_{\mathbb{T}^d} \rho_s^N\left( \frac{ \left|\sigma^{\top}_s\nabla \rho_s\right|^2}{(\rho_s)^2}\right) \, d x ds\\ 
     \end{align*}
     where in the last equality we used that 
\begin{align*}
|\sigma_s^{\top} \nabla f|^2 &=\sum_{l}    \Big|\sum_{j} \partial_jf\sigma^{jl}_s\Big|^2 = \sum_{jkl}   \partial_jf   \, ( \sigma^{jl}_s\sigma^{kl}_s) \, \partial_{k} f,
\end{align*}
and 
\begin{align*}
\left(\sigma_s^{\top} \nabla f\right)\left(\sigma_s^{\top} \nabla g\right)&=\sum_{l}\left( \left(\sum_{j} \partial_jf\sigma^{jl}_s\right)    \left(\sum_{j} \partial_jg\sigma^{jl}_s\right) \right)\\
&= \sum_{jkl}   \partial_jf   \, ( \sigma^{jl}_s\sigma^{kl}_s) \, \partial_{k} g,
\end{align*}
for $f,g \in H_2^1$.

 It follows that, 
\begin{align}
     II_t&= - \frac{1}{2} \int_0^t \int_{\mathbb{T}^d} \rho_s^N \left( \frac{\left|\sigma^{\top}_s\nabla \rho_s^N\right|^2}{(\rho_s^N)^2} -2  \frac{\left(\sigma^{\top}_s\nabla \rho_s \right)}{\rho_s} \frac{\left(\sigma^{\top}_s\nabla \rho_s^N \right)}{\rho_s^N}  + \frac{\left|\sigma^{\top}_s\nabla \rho_s\right|^2}{\rho_s^2} \right) \,dx \,ds
     \nonumber\\
    &= -\frac{1}{2} \int_0^t \int_{\mathbb{T}^d} \rho_s^N \left|\frac{\left(\sigma^{\top}_s\nabla \rho_s^N\right)}{\rho_s^N} - \frac{\left(\sigma^{\top}_s\nabla \rho_s\right)}{\rho_s}  \right|^2 \,dx \,ds
    \nonumber\\
    &= -\frac{1}{2} \int_0^t \int_{\mathbb{T}^d} \rho_s^N \left|\frac{\rho_s}{\rho_s^N}\left[\frac{\left(\sigma^{\top}_s\nabla \rho_s^N\right) \rho_s - \rho_s^N \left(\sigma^{\top}_s\nabla \rho_s\right)}{\rho_s^2}\right]\right|^2 \,dx \,ds
     \nonumber\\
    &= - \frac{1}{2} \int_0^t \int_{\mathbb{T}^d} \rho_s^N \left|\sigma^{\top}_s\frac{\rho_s}{\rho_s^N}\nabla \left(\frac{\rho_s^N}{\rho_s}\right)\right|^2 \,dx \,ds
    \nonumber\\
    &= -\frac{1}{2} \int_0^t \int_{\mathbb{T}^d} \rho_s^N \left|\sigma^{\top}_s\nabla \ln \left(\frac{\rho_s^N}{\rho_s}\right)\right|^2 \,dx \,ds. \label{II_t add}
\end{align}
Thus,  recalling (\ref{ito at entropy}) and taking $\sigma = I$ in $(\ref{II_t add})$,  we get
\begin{align}
    III_t &= 
   -\frac{1}{2} \int_0^t \int_{\mathbb{T}^d} \rho_s^N \left|\nabla \ln \left(\frac{\rho_s^N}{\rho_s}\right)\right|^2 \,dx \,ds \nonumber\\
    &= -\frac{1}{2} \int_0^t\mathcal{I}\left(\rho_s^N|\rho_s\right) \,ds, \label{II_t}
\end{align}
by definition of Fisher information.

Next, we  estimate the quadratic variation terms.
Since $\nabla \ln{f} = \frac{\nabla f}{f}$, for $f > 0$, we have 
\begin{align*}
     IV_t &\doteq  \int_{\mathbb{T}^d}\int_0^{t}\frac{1}{2}\frac{1}{\rho_s^N}\left|\sigma^{\top}_s\nabla \rho_s^N \right|^2 \, ds \,dx\nonumber\\
    &+ \frac{1}{2N^2}\sum_{i=1}^N \int_{\mathbb{T}^d} \int_0^t \frac{1}{\rho_s^N}|\nabla V^N(x - X_s^i)|^2 \,ds\,dx\nonumber \\
     &- \int_{\mathbb{T}^d}\int_0^t \frac{1}{\rho_s}  \sigma^{\top}_s\nabla \rho_s^N\sigma^{\top}_s\nabla \rho_s \, ds \, dx\nonumber\\
     &+ \int_{\mathbb{T}^d}\int_0^t \frac{1}{2}\frac{\rho_s^N}{(\rho_s)^2} \left|\sigma^{\top}_s \nabla \rho_s\right|^2 \, ds \, dx\nonumber\\
      &\overset{}{=} \int_{\mathbb{T}^d}\int_0^{t}\frac{\rho_s^N}{2}\left|\sigma^{\top}_s\nabla \ln(\rho_s^N) \right|^2 \, ds \, dx\nonumber\\
    &+ \frac{1}{2N^2}\sum_{i=1}^N \int_{\mathbb{T}^d} \int_0^t \frac{1}{\rho_s^N}|\nabla V^N(x - X_s^i)|^2 \,ds\,dx\nonumber \\
     &- \int_{\mathbb{T}^d}\int_0^t \rho_s^N  \sigma^{\top}_s\nabla \ln(\rho_s^N)\sigma^{\top}_s\nabla \ln(\rho_s) \, ds \, dx\nonumber\\
     &+ \int_{\mathbb{T}^d}\int_0^t \frac{\rho_s^N}{2} \left|\sigma^{\top}_s \nabla \ln( \rho_s)\right|^2 \, ds \, dx.\nonumber\\
\end{align*}

We observe that 
\begin{align*}
     \frac{\rho_s^N}{2}\left|\sigma^{\top}_s \nabla \ln\left(\frac{\rho_s^N}{\rho_s}\right)\right|^2 &\overset{}{=} \frac{\rho_s^N}{2}\left|\sigma^{\top}_s\nabla \ln(\rho_s^N) \right|^2 \,\nonumber\\
    &-  \rho_s^N  \sigma^{\top}_s\nabla \ln(\rho_s^N)\sigma^{\top}_s\nabla \ln(\rho_s) \, \nonumber\\
     &+  \frac{\rho_s^N}{2} \left|\sigma^{\top}_s \nabla \ln( \rho_s)\right|^2, \,\nonumber\\
\end{align*}

which implies
\begin{align}
    IV_t &=   \frac{1}{2}\int_{\mathbb{T}^d} \int_0^t \rho_s^N \left|\sigma^{\top}_s \nabla \ln\left(\frac{\rho_s^N}{\rho_s}\right)\right|^2 \,ds \, dx\nonumber\\
    &+ \frac{1}{2N^2}\sum_{i=1}^N \int_{\mathbb{T}^d} \int_0^t \frac{1}{\rho_s^N}|\nabla V^N(x - X_s^i)|^2 \,ds\,dx. \label{IV_t add}
\end{align}

Then by $(\ref{II_t add})$ and $(\ref{IV_t add})$, we obtain 
\begin{align}
     II_t + IV_t &= \frac{1}{2N^2}\sum_{i=1}^N \int_{\mathbb{T}^d} \int_0^t \frac{1}{\rho_s^N}|\nabla V^N(x - X_s^i)|^2 \,ds\,dx.\label{II + IV}
   \end{align}
   Regarding the gradient term  in (\ref{II + IV}), by  $(\ref{def Vn period})$ we get
   \begin{align*}
       \left|\nabla V^N(x - X_s^i)\right| &\overset{(\ref{def Vn period})}{\leq} \sum_{k \in \mathbb{Z}^d} \left|\nabla V_0^N\left((x - X_s^i)-k\right)\right|\\
       &= \sum_{k \in \mathbb{Z}^d} \left|N^{\frac{\beta}{d}} N^{\beta}(\nabla V_0)\left(N^{\frac{\beta}{d}}((x - X_s^i)-k)\right)\right|\\
       &= N^{\frac{\beta}{d}}\sum_{k \in \mathbb{Z}^d} N^{\beta}\left|(\nabla V_0)\left(N^{\frac{\beta}{d}}((x - X_s^i)-k)\right)\right|.
       \end{align*}
       Now we use the estimate  $|\nabla V_0| \leq C_d V_0$ given in (\ref{mollifier inequality}), which yields
       \begin{align}
        \left|\nabla V^N(x - X_s^i)\right| &\overset{(\ref{mollifier inequality})}{\leq} C_d N^{\frac{\beta}{d}} \sum_{k \in \mathbb{Z}^d} N^{\beta}\left|V_0\left(N^{\frac{\beta}{d}}((x - X_s^i)-k)\right)\right|\nonumber\nonumber\\
        &\overset{(\ref{def Vn period})}{=} C_d N^{\frac{\beta}{d}} V^N(x - X_s^i) \label{estimate grad moll}\\
        &\overset{(\ref{estimate vn})}{\leq} 2 N^{\beta}\exp{(1/4)}C_dN^{\frac{\beta}{d}}. \label{estimate grad mollifier period}
   \end{align}
   
Finally, from  $(\ref{II + IV})$,  $(\ref{estimate grad moll})$ and   $(\ref{estimate grad mollifier period})$ we deduce 
\begin{align}
    II_t + IV_t &= \frac{1}{2N^2}\sum_{i=1}^N \int_{\mathbb{T}^d} \int_0^t \frac{1}{\rho_s^N(x)}|\nabla V^N(x - X_s^i)|^2 \,ds\,dx\nonumber\\
    &\overset{(\ref{estimate grad mollifier period})}{\leq} C_d\frac{2N^{\beta}N^{\frac{\beta}{d}}}{2N^2}\sum_{i=1}^N \int_{\mathbb{T}^d} \int_0^t \frac{1}{\rho_s^N(x)}|\nabla V^N(x - X_s^i)| \,ds\,dx\nonumber\\
       &\overset{(\ref{estimate grad moll})}{\leq }C_d\frac{N^{\beta+2\frac{\beta}{d}}}{N}  \int_0^t \int_{\mathbb{T}^d} \frac{1}{\rho_s^N(x)}  \frac{1}{N}\sum_{i=1}^N V^N(x - X_s^i)\,dx\,ds\nonumber\\
       &\overset{}{ =}C_d\frac{N^{\beta+2\frac{\beta}{d}}}{N}  \int_0^t \int_{\mathbb{T}^d} \frac{1}{\rho_s^N(x)}  \rho_s^N(x)\,dx\,ds\nonumber\\
    &\leq C_d tN^{-1+\beta + 2\frac{\beta}{d}}.\label{cancelations}
\end{align}

\subsection{Estimates for $I_t$ if $K$ fulfills $(\mathbf{A}^{K})$ and $(\mathbf{A}^{\nabla\cdot K})$} \label{nonlinearnholder}

We now address the nonlinear terms in (\ref{ito at entropy}). The main idea is to use the  dissipation obtained in the previous subsection along with the assumption $(\mathbf{A}^{\nabla\cdot K})$, standard information inequalities, and the decay property (\ref{decay of mollifier}) of mollifier \(V_0\), to deal with the singularity of kernel $K$.

By integration by parts, we obtain  
\begin{align}
    I_t &\doteq 
     \int_0^t\int_{\mathbb{T}^d}  \left[\ln(\rho_s) - \ln(\rho_s^N)\right] \left\langle S_s^N, \nabla V^N(x - \cdot)K\ast \rho_s^N(\cdot)\right\rangle  \,dx \,ds\nonumber \\
    &+ \int_0^t\int_{\mathbb{T}^d}  \left[-\left\langle S_s^N, \nabla V^N(x - \cdot) K \ast \rho_s^N(\cdot) \right\rangle + \frac{\rho_s^N}{\rho_s}\nabla\cdot(\rho_s (K\ast \rho_s)) \right] \,dx \,ds \nonumber\\
    &=
     -\int_0^t\int_{\mathbb{T}^d} \nabla \ln\left(\frac{\rho_s}{\rho_s^N} \right) \left\langle S_s^N, V^N(x - \cdot)K\ast \rho_s^N(\cdot)\right\rangle  \,dx \,ds\nonumber \\
    &+ \int_0^t\underbrace{\int_{\mathbb{T}^d}  -\nabla \cdot \left\langle S_s^N,  V^N(x - \cdot) K \ast \rho_s^N(\cdot) \right\rangle \,dx}_{=0} \,ds
    \nonumber \\
    &+ \int_0^t\int_{\mathbb{T}^d}  \frac{\rho_s^N}{\rho_s}\nabla \cdot (\rho_s (K\ast \rho_s)) \,dx \,ds.
    \label{drift}
    \end{align}

   We observe that 
\begin{align*}
    \int_0^t\int_{\mathbb{T}^d}  \frac{\rho_s^N}{\rho_s}\nabla\cdot (\rho_s (K\ast \rho_s)) \,dx \,ds&=  - \int_0^t\int_{\mathbb{T}^d}  \nabla \left(\frac{\rho_s^N}{\rho_s}\right)(\rho_s K\ast \rho_s) \,dx \,ds\\
    &= - \int_0^t\int_{\mathbb{T}^d} \frac{\rho_s^N}{\rho_s^N} \nabla \left(\frac{\rho_s^N}{\rho_s}\right)(\rho_s K\ast \rho_s) \,dx \,ds\\
    &= - \int_0^t\int_{\mathbb{T}^d} \rho_s^N \nabla \ln{\left(\frac{\rho_s^N}{\rho_s}\right)}( K\ast \rho_s) \,dx \,ds.
\end{align*}
So, by subtracting and adding the term \(\rho_s^N(x)\), and noting that \(\nabla \ln{f} = \frac{\nabla f}{f}\) for \(f > 0\), applying integration by parts to equation \((\ref{drift})\) yields:
\begin{align}
    I_t &\overset{}{=}
     \int_0^t\int_{\mathbb{T}^d} \rho_s^N\nabla \ln\left(\frac{\rho_s}{\rho_s^N} \right)[K\ast \rho_s -  K\ast \rho_s^N] \,dx \,ds  \nonumber\\
     &-
     \int_0^t\int_{\mathbb{T}^d} \nabla \ln\left(\frac{\rho_s}{\rho_s^N} \right) \left\langle S_s^N, V^N(x - \cdot)[K\ast \rho_s^N(\cdot) - K\ast \rho_s^N(x)]\right\rangle \,dx \,ds  \nonumber\\
     &\doteq I^1_t - I^2_t. \label{only Holder}
\end{align}

Now we will derive an estimate for $I^1_t$. By assumption $(\mathbf{A}^{\nabla\cdot K})$, $K = \nabla\cdot K_0$ with $K_0 \in L^{\infty}$ and $\nabla\cdot K = 0$, and then { by integration by parts along with $\nabla \ln{f} = \frac{\nabla f}{f}$, $f > 0$, we obtain 

\begin{align*}
    I^1_t &\doteq \int_0^t\int_{\mathbb{T}^d} \rho_s^N\nabla \ln(\rho_s)[K\ast \rho_s -  K\ast \rho_s^N] \,dx \,ds \\
    &- \int_0^t\underbrace{\int_{\mathbb{T}^d} \rho_s^N\nabla \ln(\rho_s^N)[K\ast \rho_s  -  K\ast \rho_s^N]\,dx}_{=0 } \,ds\\
    &\overset{}{=}\int_0^t\int_{\mathbb{T}^d} \rho_s^N\nabla \ln(\rho_s)[\nabla \cdot K_0\ast \rho_s -  \nabla \cdot K_0\ast \rho_s^N] \,dx \,ds\\
     &= -\int_0^t\int_{\mathbb{T}^d} \nabla \left(\rho_s^N\nabla \ln(\rho_s) \right)[ K_0\ast \rho_s -  K_0\ast \rho_s^N]\,dx \,ds.
\end{align*}

Regarding the term in the last integral involving the gradient, by Leibnz's rule we have
\begin{align*}
    \nabla \left(\rho_s^N \nabla \ln(\rho_s)\right) &= \nabla \left(\rho_s^N \frac{\nabla \rho_s}{\rho_s}\right)\\
    &=\nabla \rho_s^N \frac{\nabla \rho_s}{\rho_s} + \rho_s^N \left[\frac{\nabla^2\rho_s \rho_s - \nabla \rho_s \nabla \rho_s}{\rho_s^2}\right]
    \\
    &=\nabla \rho_s^N \frac{\nabla \rho_s}{\rho_s} + \rho_s^N \left[\frac{\nabla^2\rho_s}{\rho_s}  -\frac{ |\nabla \rho_s|^2}{\rho_s^2}\right]
     \\
    &=\rho_s^N \frac{\nabla^2\rho_s}{\rho_s} +  \left[\nabla \rho_s^N \frac{\nabla \rho_s}{\rho_s}  -\frac{ |\nabla \rho_s|^2\rho_s^N}{\rho_s^2}\right]
    \\
    &=\rho_s^N \frac{\nabla^2\rho_s}{\rho_s} +  \nabla \rho_s\left[ \frac{\nabla \rho_s^N}{\rho_s}  -\frac{ \nabla \rho_s \rho_s^N}{\rho_s^2}\right]
    \\
    &=\rho_s^N \frac{\nabla^2\rho_s}{\rho_s} +  \nabla \rho_s\left[ \frac{\nabla \rho_s^N \rho_s - \rho_s^N \nabla \rho_s}{\rho_s^2}  \right]
    \\
    &= \left[\rho_s^N \frac{\nabla^2\rho_s}{\rho_s} + \nabla \rho_s \nabla \left(\frac{\rho_s^N}{\rho_s}\right) \right].
\end{align*}

So we get 

\begin{align}
    I^1_t&= -\int_0^t\int_{\mathbb{T}^d} \left[\rho_s^N \frac{\nabla^2\rho_s}{\rho_s}  \right][ K_0\ast \rho_s -   K_0\ast \rho_s^N] \,dx \,ds \nonumber\\
    & -\int_0^t\int_{\mathbb{T}^d} \left[ \nabla \rho_s \nabla \left(\frac{\rho_s^N}{\rho_s}\right) \right][ K_0\ast \rho_s -  K_0\ast \rho_s^N] \,dx \,ds \nonumber\\
    &\doteq - I^{1,1}_t - I^{1,2}_t. \label{It1,1,2}
\end{align}

Now by $\epsilon$-Young inequality and convolution inequality (recall that $K_0 \in L^{\infty}$), together with Theorem \ref{Teo_Krylov}, we deduce

\begin{align}
    I_t^{1,2} &=\int_0^t\int_{\mathbb{T}^d} \left[ \nabla \rho_s \left(\frac{\sqrt{\rho_s^N}}{\rho_s}\right) \left(\frac{\rho_s}{\sqrt{\rho_s^N}}\right) \nabla \left(\frac{\rho_s^N}{\rho_s}\right) \right][ K_0\ast \rho_s -  K_0\ast \rho_s^N] \,dx \,ds \nonumber \\
    &\leq \epsilon \int_0^t \int_{\mathbb{T}^d} \frac{\rho_s^2}{\rho_s^N}\left|\nabla \left(\frac{\rho_s^N}{\rho_s}\right)\right|^2\, dx \, ds + C_{\epsilon}\int_0^t \int_{\mathbb{T}^d} \frac{|\nabla \rho_s|^2}{\rho_s^2}\rho_s^N\left|K_0 \ast(\rho_s - \rho_s^N)\right|^2 \, dx \,ds \nonumber \\
    & \leq \epsilon  \int_0^t \int_{\mathbb{T}^d} \rho_s^N \left|\nabla \ln \left(\frac{\rho_s^N}{\rho_s}\right)\right|^2\, dx \, ds + C_{\epsilon} \lambda^2\|K_0\|_{\infty}^2\int_0^t \|\nabla \rho_s\|^2_{\infty} \|\rho_s - \rho_s^N\|_1^2  \,ds
    \nonumber\\
    & \leq \epsilon  \int_0^t \mathcal{I}(\rho_s^N|\rho_s) \,ds + CC_{\epsilon} \lambda^2 \|K_0\|_{\infty}^2 \int_0^t \|\nabla \rho_s\|^2_{\infty} \mathcal{H}(\rho_s^N|\rho_s) \,ds \label{It12}
\end{align}
where in the last inequality  we have used  the Lemma \ref{L1entropyfisherinformation inequality} in Appendix C.

In addition, since $ K=\nabla\cdot K_0$ and $\nabla\cdot K =0$, by subtract and add the term $\rho_s$ and convolution inequality,  we have by integration by parts 
\begin{align*} 
    I_t^{1,1}&\overset{}{\doteq} \int_0^t\int_{\mathbb{T}^d} \left[\rho_s^N \frac{\nabla^2\rho_s}{\rho_s}  \right][ K_0\ast \rho_s -   K_0\ast \rho_s^N] \,dx \,ds\\
    &=\int_0^t\int_{\mathbb{T}^d} \left[\rho_s^N - \rho_s\right]\left[\frac{\nabla^2\rho_s}{\rho_s}  \right][ K_0\ast \rho_s -   K_0\ast \rho_s^N] \,dx \,ds\\
    &+\int_0^t\int_{\mathbb{T}^d} \left[\rho_s \frac{\nabla^2\rho_s}{\rho_s}  \right][ K_0\ast \rho_s -   K_0\ast \rho_s^N] \,dx \,ds\\
    &\overset{}{\leq} \int_0^t\lambda\|\nabla^2 \rho_s\|_{\infty}\|K_0\|_{\infty}\|\rho_s^N-\rho_s\|_{1}^2 \, ds\\
    &-\int_0^t \underbrace{\int_{\mathbb{T}^d} \nabla \rho_s [K\ast \rho_s -   K\ast \rho_s^N] \,dx }_{=0}\,ds\\
    &\overset{}{=} \int_0^t\lambda\|\nabla^2 \rho_s\|_{\infty}\|K_0\|_{\infty}\|\rho_s^N-\rho_s\|_{1}^2 \, ds.
    \end{align*}
    Therefore, by applying $\epsilon$-Young inequality and Lemma \ref{L1entropyfisherinformation inequality} in Appendix C, we get 
    \begin{align}
     I_t^{1,1} &\overset{\text{Lemma} \ref{L1entropyfisherinformation inequality}}{\leq} \int_0^t C\lambda\|\nabla^2 \rho_s\|_{\infty}\|K_0\|_{\infty}\mathcal{H}(\rho_s^N|\rho_s)^{1/2}\mathcal{I}(\rho_s^N|\rho_s)^{1/2}\, ds \nonumber\\ &\overset{\text{Young}}{\leq} \int_0^t C^2 C_{\epsilon}\lambda^2\|\nabla^2 \rho_s\|_{\infty}^2\|K_0\|_{\infty}^2\mathcal{H}(\rho_s^N|\rho_s) \, ds + \epsilon\int_0^t\mathcal{I}(\rho_s^N|\rho_s) \, ds. 
   \label{It11}
\end{align}
Thus, from  $(\ref{It12})$ and $(\ref{It11})$, we have  the following estimate for the term $I_t^t$ in (\ref{only Holder}):
\begin{align}
   I_t^1&\leq C_{\epsilon,K_0,\lambda}\int_0^t \|\nabla^2 \rho_s\|_{\infty}^2\mathcal{H}(\rho_s^N|\rho_s)\, ds \nonumber \\
    &+ 2\epsilon  \int_0^t \mathcal{I}(\rho_s^N|\rho_s) \,ds +C_{\epsilon,K_0,\lambda} \int_0^t \|\nabla \rho_s\|^2_{\infty} \mathcal{H}(\rho_s^N|\rho_s)  \,ds. \label{It1}
\end{align}
Now, we will focus on the term $I_t^2$ in (\ref{only Holder}). 
 For that purpose, we first use the polynomial decay of our mollifier $V_0$ given by $(\ref{decay of mollifier})$.  
 
  Indeed, if $|x - \cdot| < N^{-\frac{\beta \gamma}{d}}$ we have 
   \begin{align}
      \left|\left\langle S^N_s, V^N(x - \cdot)|x-\cdot|^{\gamma} \right \rangle \right|\leq  N^{-\theta_2}\rho_s^N(x) \label{triangularization  without compact support 0}
  \end{align} 
  with $\theta_2 \doteq\frac{\beta \gamma^2}{d}$.
  
  Now, we consider the case that $|x - \cdot| \geq N^{-\frac{\beta \gamma}{d}}$. We  note that by $(\ref{decay of mollifier})$,
  
   \begin{align*}
       V_0\left(N^{\frac{\beta}{d}}\left((x - \cdot)-k\right)\right) & \overset{(\ref{decay of mollifier})}{\leq} \frac{C_d}{\left|N^{\frac{\beta}{d}}\right|^{2d}\left|\left((x-\cdot) -k \right)\right|^{2d}}, \, \, k \in \mathbb{Z}^d.
   \end{align*}
   For $k=0$
   \begin{align*}
      V_0\left(N^{\frac{\beta}{d}}\left((x - \cdot)\right)\right) 
       & \leq \frac{C_d}{\left|N^{\frac{\beta}{d}}\right|^{2d}\left|\left(x-\cdot \right)\right|^{2d}}\\
       &\overset{|x - \cdot| \geq N^{-\frac{\beta \gamma}{d}}}{ \leq} \frac{C_d\left|N^{\frac{\beta\gamma}{d}}\right|^{2d}}{\left|N^{\frac{\beta}{d}}\right|^{2d}} = \frac{N^{2\beta\gamma}}{N^{\beta}} \frac{C_d}{N^{\beta}}.
        \end{align*}
        For $k \neq 0$,
         \begin{align*}         \sum_{k\neq0}V_0\left(N^{\frac{\beta}{d}}\left((x - \cdot)-k\right)\right)&\leq
      N^{-2\beta}\sum_{k\neq0}\frac{C_d}{\left|\left((x-\cdot) -k \right)\right|^{2d}} \leq C_d N^{-2\beta}.
   \end{align*}
   So, if  $|x - \cdot| \geq N^{-\frac{\beta \gamma}{d}}$ and $\gamma \in (0,\frac{1}{2})$, we have 
   \begin{align*}
       V^N(x-\cdot) &\overset{(\ref{def Vn period})}{=} \sum_{k \in \mathbb{Z}^d} V_0^N((x - \cdot) - k)\\
       &=  \sum_{k \in \mathbb{Z}^d}N^{\beta}V_0\left(N^{\frac{\beta}{d}}\left((x - \cdot)-k\right)\right) \leq C_dN^{2\beta \gamma - \beta}
   \end{align*}
    and then since that $|x - \cdot| \leq 1$ we have 
     \begin{align}
      \left|\left\langle S^N_s, V^N(x - \cdot)|x-\cdot|^{\gamma} \right \rangle \right|\leq C_d N^{-\theta_1}, \label{triangularization  without compact support 1}
  \end{align}
  with $\theta_1 \doteq \beta(1-2\gamma)$.
  
  From $(\ref{triangularization  without compact support 0})$  and   $(\ref{triangularization  without compact support 1})$, we get 
  \begin{align}
      \left|\left\langle S^N_s, V^N(x - \cdot)|x-\cdot|^{\gamma} \right \rangle \right|\leq C_dN^{-\theta_1} + N^{-\theta_2}\rho_s^N(x) \label{triangularization  without compact support}
  \end{align}
  with $\theta_1 \doteq \beta(1-2\gamma)$ and $\theta_2 \doteq \frac{\beta \gamma^2}{d}$.
Now we address the difference term involving the convolution with $K$ in (\ref{only Holder}).

In fact, we recall that by assumption $(\mathbf{A}^{\nabla\cdot K})$, $K = \nabla\cdot K_0$ with $K_0 \in C^{\gamma}$.
So, we derive
\begin{align}
    \left|K \ast \rho_s^N(\cdot) - K\ast \rho_s^N(x)\right| & = \left|K_0 \ast \nabla \rho_s^N(\cdot) - K_0\ast \nabla \rho_s^N(x)\right|\nonumber\\
    &\leq \int_{\mathbb{T}^d} \left|K_0(\cdot -y) - K_0(x - y)\right||\nabla \rho_s^N(y)| \, dy\nonumber\\
    &\leq \int_{\mathbb{T}^d} \|K_0\|_{\gamma}\left|\cdot - x\right|^{\gamma}|\nabla\rho_s^N(y)| \, dy \nonumber\\
    &= \|K_0\|_{\gamma}|\cdot - x|^{\gamma} \int_{\mathbb{T}^d} |\nabla\rho_s^N(y)| \, dy. \label{K_0 Holder}
\end{align}
In addition, by $(\ref{vndensity})$, $\rho_s^N \in \mathcal{P}\left(\mathbb{T}^d\right)$, $\mathbb{P}$-a.s. and then by Holder's inequality we obtain 
\begin{align}
   \int_{\mathbb{T}^d}|\nabla \rho_s^N| \,dx &= \int_{\mathbb{T}^d}|\nabla \rho_s^N| \frac{\sqrt{\rho_s^N}}{\sqrt{\rho_s^N}} \,dx \nonumber\\
    &\overset{\text{Holder}}{\leq} \left(\int_{\mathbb{T}^d}\frac{|\nabla \rho_s^N|^2}{\rho_s^N} \,dx\right)^{\frac{1}{2}} \left(\int_{\mathbb{T}^d}\rho_s^N \,dx \right)^{\frac{1}{2}} \nonumber \\
    &\overset{(\ref{vndensity})}{=} \left(\int_{\mathbb{T}^d}\frac{|\nabla \rho_s^N|^2}{\rho_s^N}\,dx \right)^{\frac{1}{2}}. \label{fisher and L1}
\end{align}
Recalling that   $\lambda^{-1} \leq \rho \leq \lambda$, for some $\lambda >1$, $\mathbb{P}$-a.s., from  Leibniz's rule we deduce 

\begin{align} \label{fisher rho^N and difference}
   \int_{\mathbb{T}^d}  \frac{|\nabla \rho_s^N|^2}{\rho_s^N} \,dx \nonumber& = \int_{\mathbb{T}^d}  \frac{\left|\nabla \left(\rho_s\frac{\rho_s^N}{\rho_s}\right)\right|^2}{\rho_s^N} \,dx\\
    \nonumber&\leq2\int_{\mathbb{T}^d}  \frac{\left|\nabla \rho_s\left(\frac{\rho_s^N}{\rho_s}\right)\right|^2}{\rho_s^N} \,dx + 2\int_{\mathbb{T}^d}  \frac{\left|\rho_s\nabla \left(\frac{\rho_s^N}{\rho_s}\right)\right|^2}{\rho_s^N} \,dx
    \\
   &\overset{(\ref{vndensity})}{\leq} 2\lambda^2 \|\nabla \rho_s\|^2_{\infty} + 2\int_{\mathbb{T}^d} \rho_s^N \left|\nabla \ln \left(\frac{\rho_s^N}{\rho_s}\right)\right|^2 \,dx
    \nonumber\\
   &= 2\lambda^2 \|\nabla \rho_s\|^2_{\infty} + 2\mathcal{I}(\rho_s^N|\rho_s).
\end{align}
From $(\ref{K_0 Holder})$, $(\ref{fisher and L1})$ and $(\ref{fisher rho^N and difference})$, recalling $I_t^2$ in (\ref{only Holder}), we have 
\begin{align*}
   \left| \left\langle S_s^N, V^N(x - \cdot)[K\ast \rho_s^N(\cdot) - K\ast \rho_s^N(x)]\right\rangle \right| &\leq \|K_0\|_{\gamma}\left\langle S_s^N, V^N(x - \cdot)|x-\cdot|^{\gamma}\right\rangle  \\
&\times\left(2\lambda^2\|\nabla \rho_s\|_{\infty}^2 + 2\mathcal{I}\left(\rho_s^N|\rho_s\right)\right)^{\frac{1}{2}} 
\end{align*}
which implies, by $\epsilon$-Young inequality,
\begin{align*}
    I^{2}_t &\doteq \int_0^t\int_{\mathbb{T}^d} \nabla \ln\left(\frac{\rho_s}{\rho_s^N} \right) \left\langle S_s^N, V^N(x - \cdot)[K\ast \rho_s^N(\cdot) - K\ast \rho_s^N(x)]\right\rangle \,dx \,ds \\
    &\lesssim \int_0^t\int_{\mathbb{T}^d} \left| \nabla \ln\left(\frac{\rho_s}{\rho_s^N} \right)\right| \left| \left\langle S_s^N, V^N(x - \cdot)|x-\cdot|^{\gamma}\right\rangle \right| \left(\|\nabla \rho_s\|_{\infty}^2 + \mathcal{I}\left(\rho_s^N|\rho_s\right)\right)^{\frac{1}{2}}\,dx \,ds  \\
    &\lesssim \int_0^t\int_{\mathbb{T}^d} \left| \nabla \ln\left(\frac{\rho_s^N}{\rho_s} \right)\right| \left| \left\langle S_s^N, V^N(x - \cdot)|x-\cdot|^{\gamma}\right\rangle \right|^{(\frac{1}{2} + \frac{1}{2})} \left(\|\nabla \rho_s\|_{\infty}^2 + \mathcal{I}\left(\rho_s^N|\rho_s\right)\right)^{\frac{1}{2}}\,dx \,ds  \\
     &\overset{\epsilon-\text{Young}}{\leq} \epsilon \int_0^t\int_{\mathbb{T}^d} \left| \nabla \ln\left(\frac{\rho_s^N}{\rho_s} \right)\right|^2\left| \left\langle S_s^N, V^N(x - \cdot)|x-\cdot|^{\gamma}\right\rangle \right|\,dx \,ds  \\
      &+ C_{\epsilon, K_0, \lambda}\int_0^t\int_{\mathbb{T}^d}  \left| \left\langle S_s^N, V^N(x - \cdot)|x-\cdot|^{\gamma}\right\rangle \right|\left(\|\nabla \rho_s\|_{\infty}^2 + \mathcal{I}\left(\rho_s^N|\rho_s\right)\right)\,dx \,ds.  \end{align*}

      Thus by the estimate (\ref{triangularization  without compact support}) and $|x-\cdot| \leq 1$, we find 
      \begin{align*}
      I_t^2 &\overset{(\ref{triangularization  without compact support})}{\leq}\epsilon \int_0^t\int_{\mathbb{T}^d} \left| \nabla \ln\left(\frac{\rho_s^N}{\rho_s} \right)\right|^2\left| \left\langle S_s^N, V^N(x - \cdot)\right\rangle \right|\,dx \,ds  \\
      &+ C_{\epsilon, K_0, \lambda, d}\int_0^t\int_{\mathbb{T}^d} \left[ N^{-\theta_1} + N^{-\theta_2}\rho_s^N(x) \right]\left(\|\nabla \rho_s\|_{\infty}^2 + \mathcal{I}\left(\rho_s^N|\rho_s\right)\right)\,dx \,ds  \\
      &\overset{}{=}\epsilon \int_0^t\int_{\mathbb{T}^d} \rho_s^N\left| \nabla \ln\left(\frac{\rho_s^N}{\rho_s} \right)\right|^2\,dx \,ds  \\
      &+ C_{\epsilon,K_0,\lambda,d}\int_0^t\left(\|\nabla \rho_s\|_{\infty}^2 + \mathcal{I}\left(\rho_s^N|\rho_s\right)\right)\int_{\mathbb{T}^d} \left[ N^{-\theta_1} + N^{-\theta_2}\rho_s^N \right]\,dx \,ds.
      \end{align*}
      Again since that by $(\ref{vndensity})$, $\rho_s^N \in \mathcal{P}\left(\mathbb{T}^d\right)$, $\mathbb{P}$-a.s., for $N>>1$, we obtain
      \begin{align}
     I^2_t &\overset{(\ref{vndensity})}{=} \epsilon\int_0^t\mathcal{I}(\rho_s^N|\rho_s) \,ds  \nonumber\\
      &+  C_{\epsilon,K_0,\lambda,d} \left[ N^{-\theta_1} + N^{-\theta_2} \right]  \int_0^t \left(\|\nabla \rho_s\|_{\infty}^2 + \mathcal{I}\left(\rho_s^N|\rho_s\right)\right)\,ds\nonumber\\
      &\overset{N>>1}{\leq} 2\epsilon\int_0^t\mathcal{I}(\rho_s^N|\rho_s) \,ds  \nonumber\\
      &+    C_{\epsilon,K_0,\lambda,d}\left[ N^{-\theta_1} + N^{-\theta_2} \right]   \int_0^t \|\nabla \rho_s\|_{\infty}^2 \,ds. \label{It2}
      \end{align}
      
Finally, by $(\ref{It1})$ and $(\ref{It2})$ in $(\ref{only Holder})$,  $N>>1$, we conclude  with the following estimate for the nonlinear term:
\begin{align}
   I_t
   &\leq C_{\epsilon,K_0,\lambda} \int_0^t \|\nabla^2 \rho_s\|^2_{\infty} \mathcal{H}(\rho_s^N|\rho_s)  \,ds\nonumber\\
    &+ 2\epsilon  \int_0^t \mathcal{I}(\rho_s^N|\rho_s) \,ds + C_{\epsilon,K_0,\lambda} \int_0^t \|\nabla \rho_s\|^2_{\infty} \mathcal{H}(\rho_s^N|\rho_s)  \,ds\nonumber\\
    &\overset{}{+} 2\epsilon\int_0^t\mathcal{I}(\rho_s^N|\rho_s) \,ds +    C_{\epsilon,K_0,\lambda,d}\left[ N^{-\theta_1} + N^{-\theta_2} \right]   \int_0^t \|\nabla \rho_s\|_{\infty}^2 \,ds  \label{I_t regular kernel}
\end{align}
 with $\theta_1 \doteq \beta(1-2\gamma)$ and $\theta_2 \doteq \frac{\beta \gamma^2}{d}$.
\subsection{Estimates for $I_t$ if $K \in C^{\gamma}$}\label{nonlinearholder}

 When the interaction kernel $K$ is a Holder continuous function, a more straightforward estimate becomes available for the nonlinear term.

Indeed, starting from $(\ref{only Holder})$, since $\rho_s^N \in \mathcal{P}\left(\mathbb{T}^d\right)$, $\mathbb{P}$-a.s. and $K \in L^{\infty}$, by convolution inequality, $\epsilon$-Young inequality and Lemma \ref{L1entropyfisherinformation inequality} in Appendix C, 
\begin{align*}
    I_t^1 &\doteq \int_0^t\int_{\mathbb{T}^d} \rho_s^N\nabla \ln\left(\frac{\rho_s}{\rho_s^N} \right)[K\ast \rho_s -  K\ast \rho_s^N] \,dx \,ds \\
    &\overset{\text{Young}}{\leq} \epsilon\int_0^t\mathcal{I}\left(\rho_s^N|\rho_s\right) \,ds +C_{\epsilon,K}\int_0^{t}\|\rho_s^N - \rho_s\|_1^2 \,ds\\
     &\overset{\text{Lemma} \ref{L1entropyfisherinformation inequality}}{\leq} \epsilon\int_0^t\mathcal{I}\left(\rho_s^N|\rho_s\right) \,ds +CC_{\epsilon,K}\int_0^{t}\mathcal{H}\left(\rho_s^N|\rho_s\right) \,ds.
\end{align*}
Also by definition of convolution and  $\rho_s^N \in \mathcal{P}\left(\mathbb{T}^d\right)$, $\mathbb{P}$-a.s., by (\ref{vndensity})
\begin{align}
    \left|K \ast \rho_s^N(\cdot) - K\ast \rho_s^N(x)\right| 
    &\leq \int_{\mathbb{T}^d} \left|K(\cdot -y) - K(x - y)\right|| \rho_s^N(y)| \, dy\nonumber\\
    &\leq \int_{\mathbb{T}^d} \|K\|_{\gamma}\left|\cdot - x\right|^{\gamma}|\rho_s^N(y)| \, dy \nonumber\\
    &\overset{(\ref{vndensity})}{=}\|K\|_{\gamma}|\cdot - x|^{\gamma} \nonumber 
\end{align}
and then by $\epsilon$-Young inequality
\begin{align*}
    I^{2}_t &\doteq \int_0^t\int_{\mathbb{T}^d} \nabla \ln\left(\frac{\rho_s}{\rho_s^N} \right) \left\langle S_s^N, V^N(x - \cdot)[K\ast \rho_s^N(\cdot) - K\ast \rho_s^N(x)]\right\rangle \,dx \,ds \\
    &\lesssim \int_0^t\int_{\mathbb{T}^d} \left| \nabla \ln\left(\frac{\rho_s}{\rho_s^N} \right)\right| \left| \left\langle S_s^N, V^N(x - \cdot)|x-\cdot|^{\gamma}\right\rangle \right|\,dx \,ds  \\
    &\lesssim \int_0^t\int_{\mathbb{T}^d} \left| \nabla \ln\left(\frac{\rho_s^N}{\rho_s} \right)\right| \left| \left\langle S_s^N, V^N(x - \cdot)|x-\cdot|^{\gamma}\right\rangle \right|^{(\frac{1}{2} + \frac{1}{2})} \,dx \,ds  \\
     &\overset{\epsilon-\text{Young}}{\leq} \epsilon \int_0^t\int_{\mathbb{T}^d} \left| \nabla \ln\left(\frac{\rho_s^N}{\rho_s} \right)\right|^2\left| \left\langle S_s^N, V^N(x - \cdot)|x-\cdot|^{\gamma}\right\rangle \right|\,dx \,ds  \\
      &+ C_{\epsilon,K}\int_0^t\int_{\mathbb{T}^d}  \left| \left\langle S_s^N, V^N(x - \cdot)|x-\cdot|^{\gamma}\right\rangle \right|\,dx \,ds.  \end{align*}

      By estimates in (\ref{triangularization  without compact support}), $|x-\cdot| \leq 1$ and $\rho_s^N \in \mathcal{P}\left(\mathbb{T}^d\right)$, $\mathbb{P}$-a.s., we obtain 
      \begin{align*}
     I_t^2&\overset{(\ref{triangularization  without compact support})}{\leq}\epsilon \int_0^t\int_{\mathbb{T}^d} \left| \nabla \ln\left(\frac{\rho_s^N}{\rho_s} \right)\right|^2\left| \left\langle S_s^N, V^N(x - \cdot)\right\rangle \right|\,dx \,ds  \\
      &+ C_{\epsilon,K,d}\int_0^t\int_{\mathbb{T}^d} \left[ N^{-\theta_1} + N^{-\theta_2}\rho_s^N \right]\,dx \,ds  \\
      &\overset{}{=}\epsilon \int_0^t\int_{\mathbb{T}^d} \rho_s^N\left| \nabla \ln\left(\frac{\rho_s^N}{\rho_s} \right)\right|^2\,dx \,ds  \\
      &+ C_{\epsilon,K, d}\int_0^t\int_{\mathbb{T}^d} \left[ N^{-\theta_1} + N^{-\theta_2}\rho_s^N \right]\,dx \,ds  \\
      &\overset{(\ref{vndensity})}{\leq} \epsilon\int_0^t\mathcal{I}(\rho_s^N|\rho_s) \,ds  \\
      &+  C_{\epsilon,K,d} \left[ N^{-\theta_1} + N^{-\theta_2} \right] t.\\
           \end{align*}
Finally, we deduce
\begin{align}
    I_t \leq 2\epsilon\int_0^t\mathcal{I}\left(\rho_s^N|\rho_s\right) \,ds +C_{\epsilon,K}\int_0^{t}\mathcal{H}\left(\rho_s^N|\rho_s\right) \,ds + C_{\epsilon,K,d} \left[ N^{-\theta_1} + N^{-\theta_2} \right] t \label{drift holder}
\end{align}
 with $\theta_1 \doteq \beta(1-2\gamma)$ and $\theta_2 \doteq \frac{\beta \gamma^2}{d}$.
\subsection{End of the proof: application of Gronwall's Lemma
}\label{end of the proof}
 In this subsection, we conclude our estimates.
To begin, we state the following bound for the martingale term, with the corresponding proof given in Appendix A.
\begin{lem} \label{martingale}
    It holds that 
    \begin{align*} 
       \left[\mathbb{E}\left(\sup_{t\in[0,T]}|M_t^N|\right)^m \right]^{\frac{1}{m}}  \lesssim N^{-\theta_3} 
    \end{align*}
    with
   $$\theta_3 \doteq \frac{1}{2}-\beta\Big(1 + \frac{1 }{d}\Big)$$
    for all $m\geq1$ and $N \in \mathbb{N}$.
\end{lem}

We now complete the proof of the main Theorems, by applying Gronwall's inequality.
\subsubsection{Proof of Theorem \ref{first main}}

Assuming  that \( K_0 \in C^{\gamma} \), we put $(\ref{II_t}), (\ref{cancelations}),(\ref{I_t regular kernel}), $ into $(\ref{ito at entropy})$. Then, for $N>>1$ we get 
\begin{align}
    \mathcal{H}(\rho_t^N|\rho_t) - \mathcal{H}(\rho_0^N|\rho_0) &\leq C_{\epsilon,K_0,\lambda}\int_0^t\|\nabla^2 \rho_s\|_{\infty}^2\mathcal{H}(\rho_s^N|\rho_s)\, ds \nonumber\\
    &+  C_{\epsilon,K_0,\lambda} \int_0^t \|\nabla \rho_s\|^2_{\infty} \mathcal{H}(\rho_s^N|\rho_s)  \,ds\nonumber\\
    &\overset{}{+} 4\epsilon\int_0^t\mathcal{I}(\rho_s^N|\rho_s) \,ds\nonumber\\
    &+  C_{\epsilon,K_0,\lambda,d}\left[ N^{-\theta_1} + N^{-\theta_2} \right]   \int_0^t \|\nabla \rho_s\|_{\infty}^2 \,ds 
    \nonumber\\
    &-\frac{1}{2} \int_0^t \mathcal{I}(\rho_s^N|\rho_s)  \,ds \nonumber\\
    &+tC_dN^{-1+\beta + 2\frac{\beta}{d}}\nonumber\\
    &+ M_t^N. \label{entropy regular grownall}
\end{align}
By Lemma \ref{martingale} combined with Lemma \ref{borel cantelli} in Appendix C, for all $\delta \in (0,\theta_3)$ there exists a random variable, \( A_0 \) with finite moments such that
\begin{align*}       \sup_{t\in[0,T]}|M_t^N|  \leq A_0N^{-\theta_3 + \delta},
    \end{align*}
which allows us to express 
(\ref{entropy regular grownall}), for $\epsilon = \frac{1}{8}$, as
\begin{align}
    \mathcal{H}(\rho_t^N|\rho_t) - \mathcal{H}(\rho_0^N|\rho_0) &\leq C_{\epsilon,K_0,\lambda}\int_0^t\|\nabla^2 \rho_s\|_{\infty}^2\mathcal{H}(\rho_s^N|\rho_s)\, ds \nonumber\\
    &+ C_{\epsilon,K_0,\lambda} \int_0^t \|\nabla \rho_s\|^2_{\infty} \mathcal{H}(\rho_s^N|\rho_s)  \,ds\nonumber\\
    &\overset{}{+}  C_{\epsilon,K_0,\lambda,d}\left[ N^{-\theta_1} + N^{-\theta_2} \right]   \int_0^t \|\nabla \rho_s\|_{\infty}^2 \,ds 
    \nonumber\\
    &+tC_dN^{-1+\beta + 2\frac{\beta}{d}}\nonumber\\
    &+ A_0N^{-\theta_3 + \delta} \label{entropy regular grownall without expectation 1}
\end{align}
 with $\theta_1 \doteq \beta(1-2\gamma)$, $\theta_2 \doteq \frac{\beta \gamma^2}{d}$ and $\theta_3 \doteq \frac{1}{2}-\beta\Big(1 + \frac{1 }{d}\Big)$.
 
From  Gronwall's Lemma, we deduce  

\begin{align*}
   \sup_{t \in [0,T]}\mathcal{H}(\rho_t^N|\rho_t) &\lesssim  \left(\mathcal{H}(\rho_0^N|\rho_0) + N^{-\theta_4}\int_0^T \|\nabla \rho_t\|_{\infty}^2 dt + N^{-\theta_4} + A_0N^{-\theta_4} \right)\\
   &\times \exp\left(\int_0^T \|\nabla \rho_t\|_{\infty}^2 dt + \int_0^T \|\nabla^2 \rho_t\|_{\infty}^2 dt\right)\nonumber\\
   \end{align*}
   with
     \begin{align*}
\theta_4 \doteq \min \left(\beta(1-2\gamma);\frac{\beta}{d}\gamma^2;\left(\frac{1}{2}-\beta\Big(1 + \frac{1 }{d}\Big) \right) - \delta  \right).
\end{align*}

By Corollary \ref{coro Teo_Krylov} and  Sobolev embedding for $q>d$, we have 
    \begin{align*}
       \int_0^T\|\nabla \rho_t\|_{\infty}^2 \, dt, \int_0^T\|\nabla^2 \rho_t\|_{\infty}^2 \, dt < \infty, \, \, \, \mathbb{P}-a.s.
    \end{align*}  
Thus for 
\begin{align*}
\theta \doteq \min \left(\beta(1-2\gamma);\frac{\beta}{d}\gamma^2;\left(\frac{1}{2}-\beta\Big(1 + \frac{1 }{d}\Big) \right)  \right) - \delta,
\end{align*}
we end up with
    \begin{align*}
        \lim_{N \to \infty} N^{\theta}  \sup_{t \in [0,T]}\mathcal{H}(\rho_t^N|\rho_t) = 0, \, \, \, \mathbb{P}-a.s.,
    \end{align*}
    since  $\theta < \theta_4$, which conclude the proof of   Theorem \ref{first main}.
 \subsubsection{Proof of Theorem \ref{second main}}
If $K_0 \in C^{\gamma}$ and $\sigma =0$, by Jensen's inequality applied to the function $|\cdot|^m$, $m\geq 1$ and taking expectation in (\ref{entropy regular grownall}), $\epsilon=\frac{1}{8}$ implies, for $N>>1$ and $r \in [0,T]$,
\begin{align}
    \mathbb{E}\left(\sup_{t \in [0,r]}\mathcal{H}(\rho_t^N|\rho_t)\right)^m  &\leq \mathbb{E}\left(\mathcal{H}(\rho_0^N|\rho_0)\right)^m\nonumber\\
    &+C_{\epsilon,K_0,\lambda}^m \int_0^r\|\nabla^2 \rho_s\|_{\infty}^{2m}\mathbb{E}\left(\sup_{t \in [0,s]}\mathcal{H}(\rho_t^N|\rho_t)\right)^m\, ds \nonumber\\
    &+  C_{\epsilon,K_0,\lambda}^m\int_0^r \|\nabla \rho_s\|^{2m}_{\infty} \mathbb{E}\left(\sup_{t \in [0,s]}\mathcal{H}(\rho_t^N|\rho_t)\right)^m  \,ds\nonumber\\
   &+ C_{\epsilon,K_0,\lambda,d}^m\left[ N^{-\theta_1} + N^{-\theta_2} \right]^m   \left(\int_0^r \|\nabla \rho_s\|_{\infty}^2 \,ds \right)^m\nonumber\\
    &\overset{}{+} \left(C_dN^{-1+\beta + 2\frac{\beta}{d}}\right)^mT^m 
    \nonumber\\   &+\mathbb{E}\left(\sup_{t\in[0,r]}|M_t^N|\right)^m. \label{grownal sigma zero}
\end{align}
From  Lemma $\ref{martingale} $  and  Gronwall's Lemma we deduce 

\begin{align*}
    \left[\mathbb{E}\left(\sup_{t \in [0,T]}\mathcal{H}(\rho_t^N|\rho_t)\right)^m \right]^{\frac{1}{m}}  &\lesssim  \left(\left[\mathbb{E}\left(\mathcal{H}(\rho_0^N|\rho_0)\right)^m  \right]^{\frac{1}{m}} + N^{-\theta}\left(\int_0^T \|\nabla \rho_s\|_{\infty}^2 \,ds \right) + N^{-\theta} \right)\\
    &\times \exp\left(\frac{1}{m}\int_0^T \|\nabla \rho_t\|_{\infty}^{2m} dt + \frac{1}{m}\int_0^T \|\nabla^2 \rho_t\|_{\infty}^{2m} dt\right)\nonumber\\
   \end{align*}
   with
  \begin{align*}
\theta = \min \left(\beta(1-2\gamma);\frac{\beta}{d}\gamma^2;\frac{1}{2}-\beta\Big(1 + \frac{1 }{d}\Big) \right)
\end{align*}
for $N>>1$. Applying Corollary \ref{coro Teo_Krylov} and the Sobolev embedding theorem for \( q > d \), we obtain the following:

\[
\sup_{t \in [0,T]} \|\nabla \rho_t\|_{\infty}^2, \quad \sup_{t \in [0,T]} \|\nabla^2 \rho_t \|_{\infty}^2 < \infty.
\]

Thus, we conclude the proof of Theorem \ref{second main}.

\subsubsection{Proof of Theorem \ref{third main}}
If $K \in C^{\gamma}$,  we put $(\ref{II_t}), (\ref{cancelations}),(\ref{drift holder}), $ into $(\ref{ito at entropy})$ to get
\begin{align}
    \mathcal{H}(\rho_t^N|\rho_t) - \mathcal{H}(\rho_0^N|\rho_0) &\leq 2\epsilon\int_0^t\mathcal{I}\left(\rho_s^N|\rho_s\right) \,ds \nonumber+C_{\epsilon,K}\int_0^{t}\mathcal{H}\left(\rho_s^N|\rho_s\right) \,ds\\
    &+ C_{\epsilon,K,d} \left[ N^{-\theta_1} + N^{-\theta_2} \right] t \nonumber\\
    &-\frac{1}{2} \int_0^t \mathcal{I}(\rho_s^N|\rho_s)  \,ds \nonumber\\
    &+tC_dN^{-1+\beta + 2\frac{\beta}{d}}\nonumber\\
    &+ M_t^N \label{entropy regular grownall holder}
\end{align}
 with $\theta_1 \doteq \beta(1-2\gamma)$ and $\theta_2 \doteq \frac{\beta \gamma^2}{d}$.

So by Jensen's inequality applied to function $|\cdot|^m$, $m\geq 1$ and taking expectation in (\ref{entropy regular grownall holder}), $\epsilon=\frac{1}{4}$ implies, for all $r \in [0,T]$,
\begin{align}
    \mathbb{E}\left(\sup_{t \in [0,r]}\mathcal{H}(\rho_t^N|\rho_t)\right)^m  &\leq  \mathbb{E}\left(\mathcal{H}(\rho_0^N|\rho_0)\right)^m\nonumber\\
    &+C_{\epsilon,K}^m\int_0^r\mathbb{E}\left(\sup_{t \in [0,s]}\mathcal{H}(\rho_t^N|\rho_t)\right)^m\, ds \nonumber\\
    &+  C_{\epsilon,K,d}^m\left[ N^{-\theta_1} + N^{-\theta_2} \right]^m T^m\nonumber\\
    &\overset{}{+} \left(C_dN^{-1+\beta + 2\frac{\beta}{d}}\right)^mT^m
    \nonumber\\   &+\mathbb{E}\left(\sup_{t\in[0,r]}|M_t^N|\right)^m. \label{grownal holder}
\end{align}
Using Lemma $\ref{martingale}$, the Gronwall's Lemma in
 $(\ref{grownal holder})$ implies that
\begin{align*}
   \left[ \mathbb{E}\left(\sup_{t \in [0,T]}\mathcal{H}(\rho_t^N|\rho_t)\right)^m \right]^{\frac{1}{m}}  &\lesssim  \left[\mathbb{E}\left(\mathcal{H}(\rho_0^N|\rho_0)\right)^m\right]^{\frac{1}{m}} + N^{-\theta}\nonumber\\
   \end{align*}
   with
   \begin{align*}
\theta = \min \left(\beta(1-2\gamma);\frac{\beta}{d}\gamma^2;\frac{1}{2}-\beta\Big(1 + \frac{1 }{d}\Big) \right),
\end{align*}
which proof the Theorem \ref{third main}.

    \section*{Apendix A : Proof of Lemmas  \ref{martingale}} \label{proof lemmas}

 In this section we establish suitable bounds for the martingale terms appearing in our computations. The proof relies on the stochastic Fubini's Theorem together with the Burkholder–Davis–Gundy inequality.

\begin{proof} [Proof of  Lemma \ref{martingale}]   
First since  $\nabla \ln{f} = \frac{\nabla f}{f}$, $f > 0$,   by integration by parts we obtain 
   \begin{align*}
       \int_{\mathbb{T}^d} \left[\ln{\rho_s}-\ln{\rho_s^N}\right] \nabla \rho_s^N \, dx &= - \int_{\mathbb{T}^d}\frac{\rho_s^N}{\rho_s} \nabla\rho_s\, dx + \int_{\mathbb{T}^d}\frac{\nabla \rho_s^N}{\rho_s^N} \rho_s^N \, dx \\
       &=  - \int_{\mathbb{T}^d}\frac{\rho_s^N}{\rho_s} \nabla\rho_s\, dx.
   \end{align*}
    By stochastic Fubini Theorem 2.2 in \cite{Ver} we deduce 
   \begin{align*}
       \int_{\mathbb{T}^d}\int_0^t\left[\ln(\rho_s) -\ln(\rho_s^N)\right]\sigma^{\top}_s \nabla \rho_s^N \, d  B_s dx\nonumber
         &=- \int_{\mathbb{T}^d} \int_0^t\frac{\rho_s^N}{\rho_s} \sigma^{\top}_s\nabla \rho_s \, d  B_s dx . \nonumber\\
   \end{align*}
Indeed, since that for a fixed $N \in \mathbb{N}$ the integrand is a product of measurable bounded functions, and the quadratic variation of the Brownian motion satisfies $d[B]_t=dt$, the integrability conditions of the aforementioned Theorem holds.

Thus, we have  
   \begin{align*}
       M_t^N &\doteq \int_{\mathbb{T}^d}\int_0^t\left[\ln(\rho_s) -\ln(\rho_s^N)\right]\sigma^{\top}_s \nabla \rho_s^N \, d  B_s dx\nonumber\\
         & + \int_{\mathbb{T}^d} \int_0^t\frac{\rho_s^N}{\rho_s} \sigma^{\top}_s\nabla \rho_s \, d  B_s dx  \nonumber\\
       & +\int_{\mathbb{T}^d}\int_0^t\sigma^{\top}_s \nabla \rho_s^N \, d  B_s dx \left(=\int_0^t \sigma^{\top}_s\underbrace{\int_{\mathbb{T}^d} \nabla \rho_s^N \, dx}_{=0}\, d  B_s \right)\nonumber\\
     &+ \frac{1}{N} \sum_{i=1}^N \int_{\mathbb{T}^d} \int_0^t \left[\ln(\rho_s) - \ln(\rho_s^N)\right]\nabla V^N(x - X_s^i)\,dW^i_s \,dx \nonumber \\
    &- \frac{1}{N} \sum_{i=1}^N \int_{\mathbb{T}^d}\int_0^t \nabla V^N(x - X_s^i)\,dW^i_s \,dx,  \nonumber \\
      \end{align*}
which implies that
\begin{align*}
    M_t^N &= \frac{1}{N} \sum_{i=1}^N \int_{\mathbb{T}^d} \int_0^t \left[\ln(\rho_s) - \ln(\rho_s^N)\right]\nabla V^N(x - X_s^i)\,dW^i_s \,dx \nonumber \\
    &- \frac{1}{N} \sum_{i=1}^N \int_{\mathbb{T}^d}\int_0^t \nabla V^N(x - X_s^i)\,dW^i_s \,dx.  \nonumber \\
\end{align*}
Thus, by the Burkholder–Davis–Gundy inequality and the stochastic Fubini theorem, as above,  we get

\begin{align*}\mathbb{E}\left(\sup_{t\in[0,r]}|M_t^N|\right)^m &\leq\mathbb{E}\left|\frac{1}{N} \sum_{i=1}^{N}\int_{0}^{r} \int_{\mathbb{T}^d}\left[\ln{\rho_s}-\ln{\rho_s^N}\right]   ( \nabla V^{N})(x-X_{s}^{i}) \, dx \, dW_{s}^{i} \right|^{m}\\
   &+\mathbb{E}\left|\frac{1}{N} \sum_{i=1}^{N}\int_{0}^{r} \int_{\mathbb{T}^d}   ( \nabla V^{N})(x-X_{s}^{i}) \, dx \, dW_{s}^{i} \right|^{m}\\
&\overset{\text{BDG}}{\lesssim}\mathbb{E}\left|\frac{1}{N^{2}} \sum_{i=1}^{N}\int_{0}^{r} \left(\int_{\mathbb{T}^d}\left[\ln{\rho_s}-\ln{\rho_s^N}\right]  ( \nabla V^{N})(x-X_{s}^{i})\, dx\right)^2 ds \right|^{\frac{m}{2}}\\
&+\mathbb{E}\left|\frac{1}{N} \sum_{i=1}^{N}\int_{0}^{r} \left(\int_{\mathbb{T}^d} \frac{1}{N^{\frac{1}{2}}}( \nabla V^{N})(x-X_{s}^{i})\, dx\right)^2 ds \right|^{\frac{m}{2}}\\
&\doteq A+B.
\end{align*}

\vspace{1 cm}

Now by $(\ref{def Vn period})$, we have 
 \begin{align*}
    \left|\int_{\mathbb{T}^d} \frac{1}{N^{\frac{1}{2}}}( \nabla V^{N})(x-X_{s}^{i})\, dx\right|&\leq \frac{1}{N^{1/2}}\int_{\mathbb{T}^d} \left| ( \nabla V^{N})(x-X_{s}^{i}) \right|\, dx\\  
    &\overset{(\ref{def Vn period})}{\leq} \frac{1}{N^{1/2}}\int_{\mathbb{T}^d} \sum_{k \in \mathbb{Z}^d} \left|\nabla V_0^N\left((x - X_s^i)-k\right)\right|\, dx,
\end{align*}
and then
 \begin{align*}
    \left|\int_{\mathbb{T}^d} \frac{1}{N^{\frac{1}{2}}}( \nabla V^{N})(x-X_{s}^{i})\, dx\right|&\leq  \frac{1}{N^{1/2}}\sum_{k \in \mathbb{Z}^d}\int_{\mathbb{T}^d +k}  \left|\nabla V_0^N\left(y\right)\right|\, dy\\
       &= \frac{1}{N^{1/2}}\int_{\mathbb{R}^d}  \left|\nabla V_0^N\left(y\right)\right|\, dy\\
        &= \frac{1}{N^{1/2}}\int_{\mathbb{R}^d}  \left|N^{\beta}N^{\frac{\beta}{d}}(\nabla V_0)\left(N^{\frac{\beta}{d}}y\right)\right|\, dy\\
        &= \frac{1}{N^{1/2}}\int_{\mathbb{R}^d}  \left|N^{\frac{\beta}{d}}(\nabla V_0)\left(N^{\frac{\beta}{d}}y\right)\right|\,N^{\beta} dy\\
         &= \frac{N^{\frac{\beta}{d}}}{N^{1/2}}\int_{\mathbb{R}^d}  \left|(\nabla V_0)\left(y\right)\right|\, dy.\\
\end{align*}
So, we find
\begin{align*}
    B&\doteq \mathbb{E}\left|\frac{1}{N} \sum_{i=1}^{N}\int_{0}^{r} \left(\int_{\mathbb{T}^d} \frac{1}{N^{\frac{1}{2}}}( \nabla V^{N})(x-X_{s}^{i})\, dx\right)^2 ds \right|^{\frac{m}{2}}\\
    &\leq \mathbb{E}\left|\frac{1}{N} \sum_{i=1}^{N}\int_{0}^{r} \frac{N^{ \frac{2\beta}{d}}}{N}\left(\int_{\mathbb{R}^d}  \left|(\nabla V_0)\left(y\right)\right|\, dy\right)^2 ds \right|^{\frac{m}{2}}\\
    &\leq C_{V_0,T} \left(\frac{N^{ \frac{2\beta}{d}}}{N}\right)^{\frac{m}{2}}.
\end{align*}
  
Now  by $(\ref{estimate vn})$, since $$\ln\left(2 N^{\beta}\exp{(1/4)}C_d\right) \leq 2 N^{\beta}\exp{(1/4)}C_d$$ and $\ln(\lambda) \leq \lambda \leq N^{\beta}$, for $N>>1$,
\begin{align*}
  \left(\int_{\mathbb{T}^d}\left[\ln{\rho_s}-\ln{\rho_s^N}\right]  ( \nabla V^{N})(x-X_{s}^{i})\, dx\right)^2 
  &\leq C_dN^{2\beta}\left(\int_{\mathbb{T}^d}  | ( \nabla V^{N})(x-X_{s}^{i})|\, dx \right)^2
\end{align*}
and  thus
\begin{align*}
    A \leq C_d N^{m \beta} B \leq C_{V_0,T,d}\left(\frac{N^{\beta + \frac{\beta}{d}}}{N^{1/2}}\right)^{m}.
\end{align*} 
\end{proof}

\section*{Apendix B : Proof of Theorem \ref{Teo_Krylov} and Corollary \ref{coro Teo_Krylov}} \label{well-posed and regularity}

\medskip
\subsection*{Proof of Theorem \ref{Teo_Krylov}}
The proof is carried out in two stages. First, we tackle the linearized problem using the \( L^q \)-theory for SPDEs. In the second stage, a contraction argument is employed to establish the solution of the original equation as the fixed point of the solution map. The approach follows the method outlined in \cite{HuangQiu},  \cite{Josue} and \cite{Krylov}.
\begin{proof} [Proof of Theorem \ref{Teo_Krylov}]
Let \[\mathbb{B}\doteq \left\{\rho \in S_{\mathcal{F}^B}^{\infty}\left([0, T]; L^q\left( \mathbb{T}^d\right)\right) \mid \Big\|\|\rho\|_{T,q}\Big\|_{L^{\infty}(\Omega)} \le \|\rho_0\|_q \right\} \]
be with the metric $d(\rho, \rho^{\prime})=\Big\|\|\rho-\rho^{\prime}\|_{T,q}\Big\|_{L^{\infty}(\Omega)}$.

We define the operator $\mathcal{T}: \mathbb{B} \rightarrow S_{\mathcal{F}^B}^{\infty}\left([0, T] ; L^q\left(\mathbb{T}^d\right)\right)$ as follows: for each $\xi \in \mathbb{B}$, let $\mathcal{T}(\xi):=\rho^{\xi}$ be the solution to the following linear SPDE:
\begin{numcases}
\mathrm{d}\rho_t = \frac{1}{2} \Delta \rho_t + \frac{1}{2}D^2 \rho_t (\sigma \sigma^\top)_tdt - \nabla \cdot(\rho_t (K\ast \xi_t)) \, dt \nonumber\\ \hspace{28px} - \nabla \rho_t \cdot \sigma_t\, d  B_t \label{lin_SPDE} \\
\lambda^{-1}\leq \rho_0 \leq \lambda.   \nonumber
\end{numcases}

We will check that the conditions 5.1-6 of Theorem 5.1 in \cite{Krylov} hold for $n=-1$.
For the reader's convenience, we restate these conditions below using our notations:

\begin{itemize}
\item[5.1] There exist $\Lambda_1, \Lambda_2 >0$ such that
\[  \Lambda_1 |\xi|^2 \geq \sum_{i,j,k=1}^d \delta^{ik} \delta^{jk}\, \xi^i\xi^j \ge \Lambda_2 \, |\xi|^2, \]
for any $t\geq0$, and $\xi \in \mathbb{R}^d$.

\item[5.2] For any $i,j\in \{1,...,d\}$ and $\epsilon >0$, there exists $\delta > 0$, such that
\[|a^{ij}_t(x) - a^{ij}_t(y)| + |\sigma^{i}_t(x) - \sigma^{i}_t(y)| < \epsilon \]
whenever $|x-y| < \delta$, $t\geq0$, where
\begin{align*} a^{ij}_t \doteq \frac{1}{2}\sum_{k=1}^d \Big(\delta^{ik} \delta^{jk} + \sigma^{ik}_t\sigma^{jk}_t\Big) \label{aij}.
\end{align*}

\item[5.3] For any  $i,j \in \{1,...,d\}$ and $t\geq0$, the functions $a^{ij}_t$ and $\sigma^i_t$ are continuously differentiable.

\item[5.4] For any $u \in H^{1}_q$, $f(u)=\nabla \cdot( u K\ast \xi)$ is predictable as a function taking values in $H^{-1}_q$.

\item[5.5] For any $i,j \in \{1,...,d\}$ and $t\geq0$,
\[  \|a^{ij}_t\|_{C^1} +   \|\sigma^{i}_t\|_{\infty} \leq M  . \]

\item[5.6] For any $\epsilon > 0$, there exists $C_{\epsilon} > 0$, such that, for any $t\geq0$ and $u_t, v_t \in H^1_q$, we have

\[\left\|f_t(u) - f_t(v)\right\|_{-1,q} \leq \epsilon \left\|u_t - v_t\right\|_{1,q} + C_{\epsilon}\left\|u_t - v_t\right\|_{-1,q}.\] 
\end{itemize}

The Assumption 5.1 holds, since the Laplacian operator is uniformly elliptic. 
Assumptions 5.2, 5.3 and 5.4 follow, since $a$ and $\sigma$ are independent of space. 

For assumption 5.4, the processes in question   are predictable as composition of predictable functions. 

For the Assumption, 5.6, recall that $\|K\ast f\|_{\infty} \le C \|f \|_{q}$, for $q > d$. So we have for all $u_t, v_t \in {H}^1_q $, $t \geq 0$,
\begin{align*}
\left\|\nabla \cdot\big ( K\ast \xi_t(u_t - v_t) \big)\right\|_{-1,q} &=
\left\|(I-\Delta)^{-\frac{1}{2}} \nabla \cdot\big(K\ast \xi_t(u_t - v_t)\big)\right\|_{q} \\ & \le C\left\|(K\ast \xi_t) (u_t - v_t)\right\|_{q} \\ 
& \le C\left\|\xi_t\right\|_{q}\left\|(u_t - v_t)\right\|_{q}\\ 
& \le C\|\rho_0\|_{q} \left\|(u_t - v_t)\right\|_{0,q}\\ 
& \le C\left\|(u_t - v_t)\right\|_{1,q}^{\frac{1}{2}}\left\|(u_t - v_t)\right\|_{-1,q}^{\frac{1}{2}}\\ 
& \le C\epsilon\left\|(u_t - v_t)\right\|_{1,q} + C \epsilon^{-1}\left\|(u_t - v_t)\right\|_{-1,q},
\end{align*}
by the interpolation inequality and using $ \xi \in \mathbb{B}$. 

Notice that the initial condition $\rho_0 \in H^{1-\frac{2}{q}}_q(\mathbb{T}^d)$ and is deterministic, therefore, we have by Theorem 5.1 in \cite{Krylov}  that the linear SPDE (\ref{lin_SPDE}) admits a unique solution $\rho^{\xi} \in L_{\mathcal{F}^B}^q\left([0,T];H^{1}_q\left(\mathbb{T}^d\right)\right)$. In addition, since Theorem 7.1 (iii) in  holds for $q\geq 2$ (by Theorem 4.2 in \cite{Krylov}) and $\|\rho_t\|_{1}=\|\rho_0\|_{1}=1$, $\mathbb{P}$-a.s., we have  $ \rho^{\xi} \in S_{\mathcal{F}^B}^{q}\left([0, T] ; L^1 \cap L^q\left(\mathbb{T}^d\right)\right)$. Moreover, by the maximum principle (Theorem 5.12 in \cite{Krylov}), $\lambda^{-1} \leq \rho^{\xi} \leq \lambda$, $\mathbb{P}$-a.s..
\smallskip

We now check that $\rho^{\xi}$ is also in $S_{\mathcal{F}^B}^{\infty}\left([0, T] ; L^1 \cap L^q\left(\mathbb{T}^d\right)\right)$, and we drop the superscript $\xi$ to ease the computations.  Applying the It{\^o}'s formula for the $L^q$-norm of a $H^1_q$-valued process in \cite{kry},
\begin{align}
&\| \rho_t \|_{q}^q = \|\rho_0\|_{q}^q \nonumber\\&\quad-   \frac{1}{2}\int_0^t\int_{\mathbb{T}^d}  q (q - 1) | \rho_s
(x) |^{q - 2}  |  \, \nabla \rho_s (x) |^2  \, dx \, ds
\nonumber\\&\quad- \frac{1}{2} \int_0^t\int_{\mathbb{T}^d}  q (q - 1) | \rho_s
(x) |^{q - 2}  | \sigma^{\top}_s \, \nabla \rho_s (x) |^2  \, dx\, ds  \nonumber\\
&\quad - \int_0^t\int_{\mathbb{T}^d}  q | \rho_s (x) |^{q - 2} \rho_s (x)\, \nabla \rho_s(x) \left(K\ast \xi_s \right)(x) \, dx\, ds\nonumber\\
&\quad - \int_0^t\int_{\mathbb{T}^d}  q | \rho_s (x) |^{q - 2} \rho_s (x) \nabla
\rho_s (x) \cdot \sigma_s  \, d  x\, d  B_s\nonumber\\
&\quad +  \frac{1}{2} \int_0^t\int_{\mathbb{T}^d}  q (q - 1) | \rho_s
(x) |^{q - 2}  | \sigma^{\top}_s \, \nabla \rho_s (x) |^2  \, dx\, ds \nonumber\\
&\leq \|\rho_0\|_{q}^q -  \frac{1}{2}\int_0^t \int_{\mathbb{T}^d}  q (q - 1) | \rho_s
(x) |^{q - 2}  |  \, \nabla \rho_s (x) |^2 \, dx \, ds \label{q ito}
\end{align}
where in the last inequality we are using integration by parts and $\nabla\cdot K =0$.  So,
\[ \sup_{t \in [0,T]} \norm{ \rho_t}_{q} \le \norm{ \rho_0 }_{q}, \]
which implies that $\rho^{\xi} \in \mathbb{B}$. 
\\

We now show that the map $\mathcal{T}$ is a contraction.

For any $\bar{\xi}, \xi \in \mathbb{B}$, set $\delta \rho=\rho^{\bar{\xi}}-\rho^{\xi}$ and $\delta \xi=\bar{\xi}-\xi$. As before, we apply Itôs formula for the $L^q$-norm of a $H^1_q$-valued process in \cite{kry}  to $\delta \rho$ and obtain

\begin{align*}
\| \delta \rho_t \|_{q}^q  & = -\frac{1}{2}   \int_0^t\int_{\mathbb{T}^d}  q (q - 1) | \delta \rho_s
(x) |^{q - 2}  |  \, \nabla \delta \rho_s (x) |^2  \, dx\, ds\\
&\quad -\frac{1}{2}\int_0^t  \int_{\mathbb{T}^d}  q (q - 1) | \delta \rho_s
(x) |^{q - 2}  | \sigma^{\tau}_s \, \nabla \delta \rho_s (x) |^2  \, dx\, ds\\
&\quad - \int_0^t\int_{\mathbb{T}^d}  q |  \delta \rho_s (x) |^{q - 2} \delta \rho_s \, \nabla
\cdot\left[ [ \rho^{\bar{\xi}}_s (x) K\ast \bar{\xi}_s(x)] - [ \rho^{\xi}_s (x) K\ast {\xi}_s(x)] \right] \, dx\, ds\\
&\quad - \int_0^t\int_{\mathbb{T}^d}  q | \delta \rho_s (x) |^{q - 2} \delta \rho_s (x) \nabla
\delta \rho_s (x) \cdot \sigma_s\, d  x \, d  B_s \\
&\quad + \frac{1}{2}  \int_0^t\int_{\mathbb{T}^d}  q (q - 1) | \delta \rho_s
(x) |^{q - 2}  | \sigma^{\tau}_s \, \nabla \delta \rho_s (x) |^2  \, dx \, ds
\end{align*}
and then
\begin{align*}
\| \delta \rho_t \|_{q}^q  
&\leq -\frac{1}{2}   \int_0^t \int_{\mathbb{T}^d} q (q - 1) | \delta \rho_s
(x) |^{q - 2}  |  \, \nabla \delta \rho_s (x) |^2\, dx \, ds \\
& - \int_0^t \int_{\mathbb{T}^d} q |  \delta \rho_s (x) |^{q - 2} \delta \rho_s \, \nabla
\cdot\left[ [ \rho^{\bar{\xi}}_s (x) K\ast \bar{\xi}_s(x)] - [ \rho^{\xi}_s (x) K\ast {\xi}_s(x)] \right]\, dx\, ds  \\
&\doteq - I_t - II_t
\end{align*}

\noindent where in the last inequality, we are using  integration by parts and $\nabla\cdot K =0$. 
Now we subtract and add $\rho_s^{\xi}(x)$ to get
\begin{align}
- II_t &= - \int_0^t\int_{\mathbb{T}^d} q |  \delta \rho_s  |^{q - 2} \delta \rho_s \nabla 
( \delta \rho_s)  K\ast \bar{\xi}_s\, dx \, ds \nonumber \\
&\quad - \int_0^t \int_{\mathbb{T}^d}  q |  \delta \rho_s|^{q - 2} \delta \rho_s \nabla \cdot
(\rho^{\xi}_s\,K \ast  \bar{\xi}_s )\, dx \, ds \nonumber \\
&\quad +  \int_0^t \int_{\mathbb{T}^d} q |  \delta \rho_s|^{q - 2} \delta \rho_s \nabla
\cdot ( \rho^{\xi}_s  K\ast \xi_s )  \, dx \, ds\nonumber \\
&\quad \overset{\nabla\cdot K =0}{ =}-  \int_0^t \int_{\mathbb{T}^d}  q |  \delta \rho_s|^{q - 2} \delta \rho_s \nabla \cdot
(\rho^{\xi}_s\,K \ast  \bar{\xi}_s )\, dx \, ds \nonumber \\
&\quad +  \int_0^t \int_{\mathbb{T}^d} q |  \delta \rho_s|^{q - 2} \delta \rho_s \nabla
\cdot ( \rho^{\xi}_s  K\ast \xi_s )  \, dx\, ds. \nonumber 
\end{align}
So integration by parts and Assumption $(\mathbf{A}^K)$,
\begin{align}
|II_t| &\le q(q-1) \int_0^t  \int_{\mathbb{T}^d} |\delta \rho_s  |^{q - 2}  |\nabla \delta \rho_s| \,|\rho_s^\xi|\, \,|K \ast \delta \xi_s|  \, dx \, ds \nonumber \\
&\le q(q-1) \int_0^t  \int_{\mathbb{T}^d} |\delta \rho_s  |^{q - 2}  |\nabla \delta \rho_s| \,|\rho_s^\xi|\, \|\delta \xi_s\|_{q}  \, dx \, ds.
\nonumber
\end{align}
Now by $\epsilon$-Young inequality with conjugate exponents $2$ and $2$, and then with $\frac{q}{q-2}$ and $\frac{q}{2}$, we arrive at,
\begin{align}
II_t&\le \epsilon q(q-1) \int_0^t  \int_{\mathbb{T}^d} |\delta \rho_s  |^{q - 2}  |\nabla \delta \rho_s|^2  \, dx \, ds
\nonumber \\
&+ C_{\epsilon} q(q-1) \int_0^t  \|\delta \rho_s  \|_q^{q}   \, ds
\nonumber \\
&+ C_{\epsilon} q(q-1) \int_0^t    \|\rho_s^\xi\|_q^q\, \|\delta \xi_s\|_{q}^{q}  \, ds
\nonumber\\
&\le C_\epsilon q(q-1)\int_0^t \|\delta \rho_s\|_{q}^q \, ds + \epsilon \, q (q-1) \int_0^t \int_{\mathbb{T}^d} |\delta \rho_s|^{q-2}|\nabla \delta \rho_s|^2 \, dx \, ds \nonumber\\
& \quad + C_\epsilon q(q-1) \,  \, \Big\|\|\rho^{\xi}\|_{T,q}\Big\|_{\infty}^q  \int_0^t   \|\delta \xi_s\|^q_{q} \, ds.
\end{align}

\medskip

Finally, choosing $\epsilon$ small, we have
\begin{align*}
\| \delta \rho_t \|_{q}^q
&\le C \int_0^t \|\delta \rho_s\|_{q}^{q}\, ds + C \, \|\rho_0\|_q^q \int_0^t \|\delta \xi_s\|^q_{q} \, ds. 
\end{align*}

Thus, an application of Gr\"onwall's Lemma yields

\[ \Big\|\|\delta \rho\|_{T,q}\Big\|_{\infty} \le  e^{CT}  C \|\rho_0\|_q T^\frac{1}{q}\Big\|\|\delta \xi\|_{T,q}\Big\|_{\infty}  
\] 
and taking $T$ small, we find that $\mathcal{T}$ is a contraction. 
\end{proof}
\subsection*{Proof of Corollary \ref{coro Teo_Krylov}}

 The proof relies on estimates for the nonlinear term to verify condition 5.6 of Theorem 5.1 in \cite{Krylov},  along with bootstrapping arguments from \cite{Gui}.
   
\begin{proof}[Proof of Corollary \ref{coro Teo_Krylov}]
      
   By invoking the convolution inequality and the Assumption \(\nabla\cdot K = 0\), together with \(\lambda^{-1} \leq \rho \leq \lambda\), $\mathbb{P}$-a.s., we deduce that, for all $u_t,v_t \in H^2_q$, $t \geq 0$,
    \begin{align*}
        \|\nabla \cdot\left(( K\ast \rho_t)(u_t - v_t) \right)\|_{0,q}&\leq \|K\ast \rho_t\|_{\infty}\|\nabla (u_t - v_t)\|_q\\
        &\leq \|K\|_1^q \lambda \|\nabla (u_t - v_t)\|_q\\
        & \le C\left\|(u_t - v_t)\right\|_{2,q}^{\frac{1}{2}}\left\|(u_t - v_t)\right\|_{0,q}^{\frac{1}{2}}\\ 
& \le C\epsilon\left\|(u_t - v_t)\right\|_{2,q} + C \epsilon^{-1}\left\|(u_t - v_t)\right\|_{0,q},
    \end{align*}
    by interpolation inequality.
   So since that $\rho_0 \in L^1 \cap H^2_q(\mathbb{T}^d)$,  we have by Theorem 5.1 in \cite{Krylov} and Theorem \ref{Teo_Krylov},  that  \begin{align} \label{H2 regularity}
        \rho \in L_{\mathcal{F}^B}^q\left([0,T];H^{2}_q\left(\mathbb{T}^d\right)\right) \cap S_{\mathcal{F}^B}^q\left([0,T];H^{1}_q\left(\mathbb{T}^d\right)\right) \cap \mathbb{B}.
    \end{align}
    Now by the same computations like in \cite{Gui}, Lemma 1, for all $u_t,v_t \in H^3_q$
    \begin{align}
        \partial_j\left(K\ast \rho_t \nabla (u_t - v_t) \right) = \sum_{i=1}^d \left(\partial_jK^i\ast \rho_t\right)\partial_i(u_t - v_t) + \sum_{i=1}^d \left(K^i\ast\rho_t\right)\partial_{ji} (u_t - v_t), \label{ lebris a}
    \end{align}
    and, since $K=\nabla\cdot K_0$, by definition of convolution,
\begin{align*}
    \left(\partial_jK^i\ast \rho_t \right)(x)&=  \int_{\mathbb{T}^d} \partial_jK^i(x-y)\rho_t(y) \, dy\\    
    &= \int_{\mathbb{T}^d} \partial_jK^i(y)\rho_t(x-y) \, dy\\
     &= -\int_{\mathbb{T}^d} K^i(y)\partial_j\rho_t(x-y) \, dy\\
     &= \sum_{l=1}^d\int_{\mathbb{T}^d} K_0^{il}(y)\partial_{ji}\rho_t(x-y) \, dy\\
      &= \sum_{l=1}^d\left(K_0^{il}\ast \partial_{ji} \rho_t\right)(x). \\
\end{align*}
It follows that
\begin{align*}
   \sum_{i=1}^d \left(\partial_jK^i\ast \rho_t\right)\partial_i (u_t - v_t) &=  \sum_{i=1}^d \left(\sum_{l=1}^d\left(K_0^{il}\ast \partial_{ji} \rho_t \right)\right)\partial_i (u_t - v_t)\\
   &=\left(K_0\ast\partial_j\nabla \rho_t\right) \nabla (u_t - v_t).
\end{align*}
Analogously,
\begin{align*}
    \sum_{i=1}^d \left(K^i\ast\rho_t\right)\partial_{ji} (u_t - v_t) = \left(K\ast \rho_t\right) \partial_j\nabla (u_t - v_t).
\end{align*}
Thus, by these equalities in $(\ref{ lebris a})$,
\begin{align}
    \|\partial_j\left( \left(K\ast \rho_t \right) \right)\nabla (u_t - v_t)\|_{q} &\leq  \|K_0\|_{\infty}\|\partial_j \nabla \rho_t\|_1 \|\nabla (u_t - v_t)\|_{q}\nonumber\\
    &+  \|K\|_{1}\| \rho_t\|_{\infty} \|\partial_j \nabla (u_t - v_t)\|_{q} \nonumber \\
    &\leq \left[\|K_0\|_{\infty}\left\|\partial_j\nabla \rho_t\right\|_2+\|K\|_{1}\lambda\right] \left\| (u_t - v_t)\right\|_{2,q}.  \label{boot 1}
\end{align}

\vspace{.3cm}
 
\noindent Now we will obtain a uniform estimate for $\left\|\partial_j\nabla \rho_t\right\|_2$. First take $q=2$ in $(\ref{q ito})$ to get
  \begin{align}
&\| \rho_t \|_{2}^2 \leq \|\rho_0\|_{2}^2 -   \int_0^t \int_{\mathbb{T}^d}     \, |\nabla \rho_s (x) |^2 \, dx \, ds \nonumber
\end{align}
and then
 \begin{align}
&\int_0^t \| \nabla \rho_s  \|_2^2 \, ds \leq \|\rho_0\|_{2}^2 - \| \rho_t \|_{2}^2 \leq \lambda^2.   \label{grad uniform estimate}
\end{align}

Given a multi-index \(\alpha_1\), by differentiating equation \((\ref{SPDE_Ito})\) and  applying the It{\^o}'s formula for the $L^2$-norm of a $H^1_2$-valued process in \cite{kry}, we obtain
\begin{align}
&\| \partial^{\alpha_1} \rho_t \|_{2}^2 = \|\partial^{\alpha_1}\rho_0\|_{2}^2 \nonumber\\&\quad-\int_0^t \sum_{\alpha_2}\int_{\mathbb{T}^d} | \  \partial^{\alpha_1,\alpha_2}\rho_s (x) |^2 \, dx \, ds
\nonumber\\&\quad-\int_0^t\int_{\mathbb{T}^d}  |\sigma^{\top}_s \, \nabla (  \partial^{\alpha_1}\rho_s (x) )|^2  \, dx\, ds \nonumber\\
&\quad - 2\int_0^t\int_{\mathbb{T}^d}    \partial^{\alpha_1}\rho_s (x)\,  \partial^{\alpha_1}\left(\nabla\rho_s (x)  \left(K\ast \rho_s(x) \right) \right) \, dx\, ds\nonumber\\
&\quad - 2\int_0^t \int_{\mathbb{T}^d}   \partial^{\alpha_1}\rho_s (x) \nabla
 (\partial^{\alpha_1}\rho_s (x) )\cdot \sigma_s \, d  x\, d  B_s \nonumber\\
&\quad +  \int_0^t \int_{\mathbb{T}^d}  |\sigma^{\top}_s \, \nabla (  \partial^{\alpha_1}\rho_s (x) )|^2  \, dx \, ds\nonumber\\
&\leq \|\partial^{\alpha_1}\rho_0\|_{2}^2 \nonumber\\
&+2 \int_0^t   \int_{\mathbb{T}^d}\partial^{\alpha_1,\alpha_1}\rho_s (x)\,  \left(\nabla\rho_s (x)  \left(K\ast \rho_s(x) \right) \right)\, dx \, ds\nonumber\\
&-  \int_0^t \sum_{\alpha_2}\int_{\mathbb{T}^d} | \  \partial^{\alpha_1,\alpha_2}\rho_s (x) |^2 \, dx \, ds. \label{q 2 ito derivative}
\end{align}
Using Holder's and Young inequalities,
\begin{align*}      \int_{\mathbb{T}^d}\partial^{\alpha_1,\alpha_1}\rho_s (x)\,  \left(\nabla\rho_s (x)  \left(K\ast \rho_s \right) \right)\, dx & \overset{\text{Holder}}{\leq}\|\partial^{\alpha_1,\alpha_1}\rho_s\|_2\|\left(\nabla\rho_s \left(K\ast \rho_s \right) \right)\|_2\\
&\overset{\text{Young}}{\leq}\frac{1}{4} \sum_{\alpha_2}\|\partial^{\alpha_1,\alpha_2}\rho_s\|_2^2\\
&+ \|\left(\nabla\rho_s \left(K\ast \rho_s \right) \right)\|_2^2
\end{align*}
and by convolution inequality,
\begin{align*}
    \|\left(\nabla\rho_s \left(K\ast \rho_s \right) \right)\|_2^2 \leq  \|\nabla\rho_s \|_2^2 \|K\ast \rho_s \|_{\infty}^2 \leq \|\nabla\rho_s \|_2^2 \|K \|_{1}^2 \lambda^2.
\end{align*}
So $(\ref{q 2 ito derivative})$ gives
\begin{align}
\| \partial^{\alpha_1} \rho_t \|_{2}^2  +  \frac{1}{2}\int_0^t \sum_{\alpha_2}\int_{\mathbb{T}^d} | \  \partial^{\alpha_1,\alpha_2}\rho_s (x) |^2 \, dx \, ds
&\leq \|\partial^{\alpha_1}\rho_0\|_{2}^2 \nonumber\\
&+ 2\int_0^t \|\nabla\rho_s \|_2^2 \|K \|_{1}^2 \lambda^2 \, ds \nonumber\\
&\overset{(\ref{grad uniform estimate})}{\leq} \|\partial^{\alpha_1}\rho_0\|_{2}^2 \nonumber\\
&+ 2\lambda^2 \|K \|_{1}^2 \lambda^2. \label{grad 2 uniform estimate}
\end{align}

Now observe that, through similar computations to those that led to (\ref{boot 1}), 
\begin{align}
    \|\partial^{\alpha_1}\left(\nabla\rho_s \left(K\ast \rho_s \right) \right)\|_2^2 &\leq 2 \|K_0\|_{\infty}^2\|\partial^{\alpha_1} \nabla \rho_s\|_1^2 \|\nabla \rho_s\|_2^2+ 2 \|K\|_{1}^2\| \rho_s\|_{\infty}^2 \|\partial^{\alpha_1}\nabla \rho_s\|_2^2. \label{lebris}
\end{align}
Take $(\zeta_{\eta})_{\eta\geq0}$ such that $\|\zeta_{\eta}\|_1=1$, whose compact support are assumed to be strictly
contained within $\left[-\frac{1}{2},\frac{1}{2}\right]^{d}$ and define $f = \nabla \rho(K\ast\rho)$. Note that by Definition \ref{SPDE_Ito}, $\mathbb{P}$-a.s., for all $t \in [0,T]$,
\begin{align}
    \rho^{\eta}_t(x) &= \rho_0^{\eta}(x) - \int_0^t f_s^{\eta}(x) \, ds + \frac{1}{2}\int_0^t D^2 \rho_s^{\eta}(x) (\sigma \sigma^{\top})_s \, ds + \frac{1}{2}\int_0^t \Delta \rho_s^{\eta}(x) \, ds\nonumber\\
    &-\int_0^t \sigma^{\top}_s\nabla \rho^{\eta}_s(x) \, d B_s. \label{mollification}
\end{align}
Given the multi-indices \(\alpha_1\) and \(\alpha_2\), by differentiating equation \((\ref{mollification})\)  and applying Itô's formula, we arrive at
\begin{align}
&\| \partial^{\alpha_1,\alpha_2} \rho_t^{\eta} \|_{2}^2 = \|\partial^{\alpha_1,\alpha_2}\rho_0^{\eta}\|_{2}^2 \nonumber\\&\quad-   \int_0^t \sum_{\alpha_3}\int_{\mathbb{T}^d} | \  \partial^{\alpha_1,\alpha_2,\alpha_3}\rho_s^{\eta} (x) |^2 \, dx \, ds
\nonumber\\&\quad-\int_0^t\int_{\mathbb{T}^d}  |\sigma^{\top}_s \, \nabla (  \partial^{\alpha_1,\alpha_2}\rho_s^{\eta} (x) )|^2  \, dx\, ds\nonumber\\
&\quad - 2\int_0^t\int_{\mathbb{T}^d}    \partial^{\alpha_1,\alpha_2}\rho_s (x)\,  \partial^{\alpha_1,\alpha_2}\left(f_s^{\eta}(x) \right) \, dx\, ds\nonumber\\
&\quad - 2\int_0^t\int_{\mathbb{T}^d}    \partial^{\alpha_1,\alpha_2}\rho_s^{\eta} (x) \nabla
 (\partial^{\alpha_1,\alpha_2}\rho_s^{\eta} (x) ) \sigma_s^{\top}\, d  x \, d  B_s \nonumber\\
&\quad +   \int_0^t\int_{\mathbb{T}^d}  |\sigma^{\top}_s \, \nabla (  \partial^{\alpha_1,\alpha_2}\rho_s^{\eta} (x) )|^2\, dx \, ds,  \nonumber
\end{align}
which implies that
\begin{align}
\| \partial^{\alpha_1,\alpha_2} \rho_t^{\eta} \|_{2}^2 
&\leq \|\partial^{\alpha_1,\alpha_2}\rho_0^{\eta}\|_{2}^2 \nonumber\\
&+2 \int_0^t   \int_{\mathbb{T}^d}\partial^{\alpha_1,\alpha_2,\alpha_2}\rho_s^{\eta} (x)\,  \partial^{\alpha_1}\left(f_s^{\eta}(x) \right)\, dx \, ds\nonumber\\
&-   \int_0^t \sum_{\alpha_3}\int_{\mathbb{T}^d} | \  \partial^{\alpha_1,\alpha_2,\alpha_3}\rho_s^{\eta} (x) |^2 \, dx \, ds \nonumber\\
&\overset{\text{Young}}{\leq} \|\partial^{\alpha_1,\alpha_2}\rho_0^{\eta}\|_{2}^2 \nonumber\\
&+2 \int_0^t   \int_{\mathbb{T}^d}\left|\partial^{\alpha_1}\left(f_s^{\eta}(x) \right)\right|^2\, dx \, ds\nonumber\\
&-   \frac{1}{2}\int_0^t \sum_{\alpha_3}\int_{\mathbb{T}^d} | \  \partial^{\alpha_1,\alpha_2,\alpha_3}\rho_s^{\eta} (x) |^2 \, dx \, ds. \nonumber
\end{align}
It follows that, since $\|\zeta_{\eta}\ast h\|_2 \leq \|h\|_2$, by $(\ref{lebris})$ 
\begin{align*}
\| \partial^{\alpha_1,\alpha_2} \rho_t^{\eta} \|_{2}^2 
&\leq \|\partial^{\alpha_1,\alpha_2}\rho_0\|_{2}^2 \nonumber\\
&+\int_0^t    4\|K_0\|_{\infty}^2\|\partial^{\alpha_1} \nabla \rho_s\|_1^2 \|\nabla \rho_s\|_2^2+  4\|K\|_{1}^2\| \rho_s\|_{\infty}^2 \|\partial^{\alpha_1}\nabla \rho_s\|_2^2 \, ds\nonumber\\
&-   \frac{1}{2}\int_0^t \sum_{\alpha_3}\int_{\mathbb{T}^d} | \  \partial^{\alpha_1,\alpha_2,\alpha_3}\rho_s^{\eta} (x) |^2 \, dx \, ds.\nonumber
\end{align*}
Thus from $(\ref{grad 2 uniform estimate})$, we get 
\begin{align*}
\| \partial^{\alpha_1,\alpha_2} \rho_t^{\eta} \|_{2}^2 + \frac{1}{2}\int_0^t \sum_{\alpha_3}\int_{\mathbb{T}^d} | \  \partial^{\alpha_1,\alpha_2,\alpha_3}\rho_s^{\eta} (x) |^2 \, dx \, ds 
&\leq \|\partial^{\alpha_1,\alpha_2}\rho_0\|_{2}^2 + C_{K,K_0,\Lambda,\rho_0},
\end{align*}
and since that $ \zeta_{\eta}\ast h \to h$ in $L^p$ and $\rho_0 \in L^1 \cap H^2_q(\mathbb{T}^d)$,
\begin{align}
    \| \partial^{\alpha_1,\alpha_2} \rho_t \|_{2}^2 \leq C_{K,K_0,\Lambda,\rho_0}. \label{bound second derivative rho}
\end{align}
   Finally $(\ref{bound second derivative rho})$ and interpolation inequality in 
 (\ref{boot 1}) leads to
   \begin{align}
    \|\partial_j\left( \left(K\ast \rho_t \right) \right)\nabla (u_t - v_t)\|_{q} &\leq \left[\|K_0\|_{\infty}C+\|K\|_{1}\lambda\right] \left\| (u_t - v_t)\right\|_{2,q}   \nonumber\\
      & \le C\left\|(u_t - v_t)\right\|_{3,q}^{\frac{1}{2}}\left\|(u_t - v_t)\right\|_{1,q}^{\frac{1}{2}}\nonumber\\ 
& \le C\epsilon\left\|(u_t - v_t)\right\|_{3,q} + C \epsilon^{-1}\left\|(u_t - v_t)\right\|_{1,q} \nonumber
\end{align}
   and applying   Theorem 5.1 in \cite{Krylov} and Theorem \ref{Teo_Krylov},  we have
    \begin{align} 
         \rho \in L_{\mathcal{F}^B}^{q}\left([0,T];H^{3}_{q}\left(\mathbb{T}^d\right)\right) \cap S_{\mathcal{F}^B}^{q}\left([0,T];H^{2}_{q}\left(\mathbb{T}^d\right)\right) \cap \mathbb{B}. \nonumber
    \end{align}
    
\vspace{.3cm}

\noindent Using the same calculations as above, for all $u_t,v_t \in H^4_q$,
\begin{align}
    \|\partial_{jk}\left( \left(K\ast \rho_t \right) \nabla(u_t- v_t)\right)\|_{q} &\leq  \|K_0\|_{\infty}\|\partial_{jk}\nabla \rho_t\|_2 \|\nabla (u_t - v_t)\|_{q}\nonumber\\
    &+  \|K\|_{1}\| \rho_t\|_{\infty} \|\partial_{jk} \nabla (u_t - v_t)\|_{q} \nonumber \\
     &\leq \left[\|K_0\|_{\infty}C+\|K\|_{1}\lambda\right] \left\| (u_t - v_t)\right\|_{3,q} \nonumber\\
      & \le C\left\|(u_t - v_t)\right\|_{4,q}^{\frac{1}{2}}\left\|(u_t - v_t)\right\|_{2,q}^{\frac{1}{2}}\nonumber\\ 
& \le C\epsilon\left\|(u_t - v_t)\right\|_{4,q} + C \epsilon^{-1}\left\|(u_t - v_t)\right\|_{2,q}, \nonumber
    \end{align}
by interpolation inequality.   Since $\rho_0 \in L^1 \cap H^3_q\left(\mathbb{T}^d\right)$,  we have by Theorem 5.1 in \cite{Krylov} and Theorem \ref{Teo_Krylov}, that  
    \begin{align} \label{H4 regularity}
         \rho \in L_{\mathcal{F}^B}^{q}\left([0,T];H^{4}_{q}\left(\mathbb{T}^d\right)\right) \cap S_{\mathcal{F}^B}^{q}\left([0,T];H^{3}_{q}\left(\mathbb{T}^d\right)\right) \cap \mathbb{B}. \nonumber
    \end{align}
      \end{proof}
  \medskip

\section*{Apendix C : Inequalities} \label{ineq}

We begin by establishing an inequality that relates the relative entropy functional and the Fisher information (see \cite{pedentropy}, Theorem $A.2$).
\begin{lem} \label{L1entropyfisherinformation inequality}
    It holds for $f,g \in \mathcal{P}\left(\mathbb{T}^d\right)$, $f,g > 0$, $g \in (\lambda^{-1},\lambda)$, $\lambda >1$,
    \begin{align*}
     \|f - g\|_1^2\lesssim  \mathcal{H}(f|g) \lesssim \int_{\mathbb{T}^d} f \left|\nabla \ln \left(\frac{f}{g}\right)\right|^2 \,dx = \int_{\mathbb{T}^d} \frac{g^2}{f}\left|\nabla \left(\frac{f}{g}\right)\right|^2 \,dx \doteq \mathcal{I}(f|g).
    \end{align*}
\end{lem}


	The next result is an application of the classical Borel-Cantelli's Lemma (Lemma 2.1 in \cite{Kloeden Neuenkirch}).  
\begin{lem} \label{borel cantelli}
Let $\rho > 0$ and $C(m) \in [0,\infty)$ for $m \geq1$. In addition, let $(Z^N)_{N \in \mathbb{N}}$, be a sequence of random variables such that
\begin{align*}
\left(\mathbb{E}\left|Z^N\right|^m\right)^{\frac{1}{m}} \leq C(m) N^{-\rho},
\end{align*}
for all $m \geq1$ and $N \in \mathbb{N}$. Then for all $\delta >0$, there exists a random variable $A_{\delta}$ such that almost surely
\begin{align*}
\left|Z^N\right| \leq A_{\delta}N^{-\rho + \delta}.
\end{align*}
Moreover, 
\begin{align*}
\mathbb{E}\left|A_{\delta}\right|^m < \infty,
\end{align*}
for all $m \geq1$.
\end{lem}	
	
\section*{Acknowledgements}
C. Olivera  is partially supported by  FAPESP-ANR by the  grant  $2022/03379-0$ ,  by FAPESP by the grant  $2020/04426-6$, and  CNPq by the grant $422145/2023-8$. A. B. de Souza is partially supported by  Coordenação de Aperfeiçoamento de Pessoal de Nível Superior – Brasil (CAPES) – Finance Code $001$.

\section*{Statement}
The authors declare that  have not conflicts of interest.\\
The authors made the same contributions to the manuscript.


\begin{thebibliography}{99}
\footnotesize



\bibitem{BogachevII}
V.~I. Bogachev,
\newblock \textit{Measure theory},  volume {II}, Springer-Verlag, Berlin, 2007.

\bibitem{Bres}
D.~Bresch, P.E. Jabin, and Z.~Wang.
{\em Mean-field limit and quantitative estimates with singular attractive
  kernels},  Duke Mathematical Journal, 172 (13) pp. 2591--2641, 2023.


 




\bibitem{Carmona}
 R.~Carmona and F.~Delarue, {\em Probabilistic theory of mean field games
with applications {I:} Mean field FBSDEs, control, and games}, volume~84 of Probability Theory and Stochastic
Modelling, Springer, Cham, 2018.
\newblock 

\bibitem{Carmona2}
\leavevmode\vrule height 2pt depth -1.6pt width 23pt, {\em Probabilistic theory
of mean field games with applications {II:} Mean field games with common noise and master equations}, volume~84 of Probability Theory
and Stochastic Modelling, Springer, Cham, 2018.




\bibitem{Carri}
J.A. Carrillo, X. Feng, S. Guo, P.E. Jabin, {\em Relative entropy method for particle approximation of the Landau equation for Maxwellian molecules},   arXiv:2408.15035v2, 2024.



\bibitem{Carrillo2014}
 J.A. ~Carrillo, Y. ~Choi, 
and M. ~Hauray, {\em The derivation of swarming models: Mean-field limit and Wasserstein distances}, Collective Dynamics from Bacteria to Crowds: An Excursion Through Modeling, Analysis and Simulation, pp. 1--46, 2014.

\bibitem{Ansgar}
 L.~{Chen}, E.~S. {Daus}, A.~{Holzinger}, and A.~{J{\"u}ngel}, {\em
  {Rigorous derivation of population Cross-Diffusion systems from moderately
  interacting particle systems}}, Journal of Nonlinear Science, 31, 94, 2021.
  
\bibitem{Chen2}   
L. Chen, P. Nikolaev and D.J. Promel,    \textit{Hegselmann-Krause model with environmental noise},  Transactions of American Mathematical Society, 378, pp. 527--567, 2025. 

\bibitem{Chen}
L.~Chen, A.~Holzinger, and X.~Huo, {\em {Quantitative convergence in
  relative entropy for a moderately interacting particle system on $\mathbb{R}^{d}$}}, Electronic Journal of Probability, 30, pp. 1--24, 2025.
   



  
\bibitem{Cho}
A. Chodron de Courcel, M. Rosenzweig,  S.  Serfaty, {\em The attractive log gas: stability, uniqueness, and propagation of chaos},   arXiv:2311.14560, 2023.

\bibitem{Cho2}
A. Chodron de Courcel, M. Rosenzweig,  S.  Serfaty, {\em Sharp uniform-in-time mean-field convergence for singular periodic riesz flows},  arXiv:2304.05315, 2023.




  
  


  
\bibitem{Cogui} M. Coghi,  F. Fladoli, {\it Propagation of chaos for interacting particles subject to environmental  noise}, Annals of Applied Probability, 26, pp. 1407--1442, 2016. 



\bibitem{Maurelli} M. Coghi,  M. Maurelli, \textit{Regularized vortex approximation for 2D Euler equations with transport noise},
Stochastics and Dynamics, 20, 2020.

	
\bibitem{correa}
J. Correa, C. Olivera, {\it  From particle systems to the stochastic compressible Navier-Stokes equations of a barotropic fluid},  Journal of Nonlinear Science, 35, 50, 2025.

 
\bibitem{Font}
R.~Cortez and J.~Fontbona.
\newblock \textit{Quantitative uniform propagation of chaos for {M}axwell molecules}.
\newblock Communications in Mathematical Physics, 357, 2018.

\bibitem{FlandoliLeimbachOlivera}
 F.~Flandoli, M.~Leimbach, and C.~Olivera, {\em {Uniform convergence of
  proliferating particles to the FKPP equation}}, Journal of Mathematical Analysis and Applications, 473, pp.~27--52,
  2019.

\bibitem{FlandoliLeocata}
F.~Flandoli and M.~Leocata, {\em A particle system approach to
  aggregation phenomena}, Journal of Applied Probability, 56, pp.~282--306, 2019.

\bibitem{Leocata}
F.~Flandoli, M.~Leocata and C. Ricci,  {\em The
  {N}avier-{S}tokes-{V}lasov-{F}okker-{P}lanck system as a scaling limit of
  particles in a fluid}, Journal of  Mathematical Fluid Mechanics, 23, 2021.

\bibitem{FlandoliOliveraSimon}
F.~Flandoli, C.~Olivera, and M.~Simon, {\em Uniform approximation of 2
  dimensional {N}avier-{S}tokes equation by stochastic interacting particle
  systems}, SIAM Journal on Mathematical  Analysis, 52, pp.~5339--5362, 2020.

  
\bibitem{Luo2}

  F. Flandoli, D. Luo, {\em  Mean field limit of point vortices with environmental noises to deterministic 2D Navier–Stokes equations},  Communications in Mathematics and Statistics, 2024. 

   
\bibitem{Live}
F.~Flandoli, M.~Ghio, and G.~Livieri, {\em N-player games and mean field
  games of moderate interactions}, Applied Mathematics and Optimization, 85, 2022.

\bibitem{folland2013real}
G.B. ~Folland, {\em Real analysis: Modern techniques and their applications}, in Pure and Applied Mathematics: A Wiley Series of Texts, Monographs and Tracts, Wiley, 2013.


\bibitem{FournierJourdain}
N.~Fournier and B.~Jourdain,
\newblock {\textit{Stochastic particle approximation of the Keller--Segel equation and
two-dimensional generalization of Bessel processes}}, Annals of Applied  Probability, 27,
2017.



\bibitem{FHM}
N.~Fournier, M.~Hauray and S.~Mischler.
\newblock \textit{Propagation of chaos for the 2{D} viscous vortex model},
\newblock Journal of the European Mathematical Society, 16\penalty0 (7), pp.
1423--1466, 2014.


\bibitem{Tardy}  
N. Fournier, Y. Tardy,  {\it Collisions of the supercritical Keller–Segel particle system}, 
Journal of the European Mathematical Society, 2024. 







\bibitem{Gui}  
A. Guillin, P. Le Bris and P. Monmarche, \textit{Uniform in time propagation of chaos for the 2D vortex model and other singular stochastic systems}, 
Journal of the European Mathematical Society, 2024. 

\bibitem{Hao}  
Z. Hao, J.F. Jabir, S. Menozzi and M. Rockner, {\em Propagation of chaos for moderately interacting particle systems related to singular kinetic Mckean-Vlasov SDEs}, 
Preprint   arXiv:2405.09195, 2024.  


\bibitem{quasibiot}
 M. ~Hauray, {\em Wasserstein distances for vortices approximation of Euler-type equations}, Mathematical Models and Methods in Applied Sciences, 19, pp.~6--31, 2019.




\bibitem{HuangQiu}  
  H. Huang and J. Qiu,   {\em The microscopic derivation and well-
posedness of the stochastic Keller-Segel equation},  Journal of Nonlinear Science,
 31, 2021.

\bibitem{Jabin_2018}
 P. ~Jabin,  and Z. ~Wang, {\em Quantitative estimates of propagation of chaos for stochastic systems with $W^{1,\infty}$ kernels}, Inventiones mathematicae, 214, pp.~523--591, 2018.

\bibitem{Jabin2}
P.E.  Jabin, Z.  Wang, {\em Mean field limit for stochastic particle systems}, In: Bellomo, N., Degond, P., Tadmor, E. (eds) Active Particles, Volume 1 . Modeling and Simulation in Science, Engineering and Technology, Birkhäuser, Cham. pp. 379–-402, 2017. 




\bibitem{JourdainMeleard}
B.~Jourdain and S.~M\'{e}l\'{e}ard, {\em Propagation of chaos and
  fluctuations for a moderate model with smooth initial data}, Annales de l'Institut Henri Poincaré, Probabilités et Statistiques, 34, pp.~727--766, 1998.

  

            
\bibitem{pedentropy}
 A.~ J{\"u}ngel, {\em Entropy methods for diffusive partial differential equations}, in Springer Briefs in Mathematics, Springer International Publishing, 2016.


\bibitem{Kloeden Neuenkirch}
 P. Kloeden and A. Neuenkirch, {\em The pathwise convergence of approximation schemes for stochastic differential equations}, LMS Journal of Computation and Mathematics, 10, pp.235--253, 2007.

\bibitem{Josue}
J. Knorst, C. Olivera, A.B. de Souza,  {\em Quantitative particle approximation of nonlinear stochastic Fokker-Planck equations with singular kernel}, Journal of Differential Equations, 455,  2025.

\bibitem{Kotelenez}
P. Kotelenez,   {\it  A Stochastic  Navier-Stokes equation for the vorticity of a two-dimensional fluid}, Annals of Applied Probability, 5, 1995.


\bibitem{Krylov}
N. V. ~Krylov, {\em An analytic approach to SPDEs}, In Stochastic partial differential equations: six perspectives,  
American Mathematical Society, 64, Providence, RI, 1999.

\bibitem{kry}
 N. V. ~Krylov, {\em It\^o's formula for the $L_{p}$-norm of stochastic $W^{1}_{p}$-valued processes}, Probability Theory and Related Fields, 147, pp.~583--605, 2010.


\bibitem{Lac}
D. Lacker and  L. Le Flem, {\it n-Closed-loop convergence for mean field games with common noise}, 
The Annals of Applied Probability, 2023.

\bibitem{Ansgar2}
 A.~J. Li~Chen, A.~Holzinger and N.~Zamponi, {\em Analysis and
  mean-field derivation of a porous-medium equation with fractional diffusion},
  Communications in Partial Differential Equations,  47(11), pp. 2217–2269, 2022.

\bibitem{Meleard}
S.~M\'{e}l\'{e}ard and S.~Roelly-Coppoletta, {\em A propagation of chaos
  result for a system of particles with moderate interaction}, Stochastic
  Process and their Applications, 26, pp.~317--332, 1987.

\bibitem{Nguyen}  
Q.H. Nguyen, M. Rosenzweig, S. Serfaty,  
{\it Mean-field limits of Riesz-type singular flows}, Ars Inveniendi Analytica, 2022.

  \bibitem{Niko}
P. Nikolaev, \textit{Quantitative relative entropy estimates for interacting particle systems with common noise}, arXiv:2407.01217, 2024.  


\bibitem{Oelschlager85}
 K.~Oelschl{\"a}ger, {\em A law of large numbers for moderately
  interacting diffusion processes}, Zeitschrift f{\"u}r
  Wahrscheinlichkeitstheorie und verwandte Gebiete, 69, pp.~279--322, 1985.
  

\bibitem{Oelschlager84}
 K.~Oelschl{\"a}ger, 
{\em    A martingale approach to the law of large numbers for weakly interacting stochastic processes}, Annals of Probability, 12 (2), pp. 458--479, 1984.

 \bibitem{Oelschlager87}
K. Oelschlager,  {\em A fluctuation theorem for moderately interacting diffusion processes}, Probability theory and related fields,  74(4), pp. 591–616, 1987.


\bibitem{Pisa}
C.~Olivera, A.~Richard, and M.~Toma{\v s}evi{\'c}, {\em Quantitative
  particle approximation of nonlinear {Fokker-Planck} equations with singular
  kernel}, Annali della Scuola Normale Superiore di Pisa Cl. Sci. (5), pp. 691--749, 2023.

\bibitem{Simon2}
C. Olivera and M. Simon, {\em Microscopic derivation of non-local models with anomalous diffusions from stochastic particle systems}, 253, Nonlinear Analysis,  2025.

\bibitem{Serfaty}
S.~Serfaty,
{\em Mean field limit for {C}oulomb-type flows},
Duke Mathematical Journal, 169\penalty0 (15), pp. 2887--2935, 2020.

\bibitem{Shao}
Y. Shao and X. Zhao, {\em Quantitative particle approximations of stochastic 2d navier-stokes equation}, arXiv:2402.02336, 2024. 

\bibitem{Simon}
M. Simon,   C. Olivera, {\em Non-local conservation law from stochastic particle systems}, Journal Dynamics and  Differential Equations,  30, pp. 1661–1682, 2018.

\bibitem{Sz}
A. S. Sznitman, {\em Topics in propagation of chaos}. In Paul-Louis Hennequin, editor, Ecole d’Ete de Probabilites de Saint-Flour XIX— 1989, Springer Berlin Heidelberg, pp.165–251, 1991. 



 

\bibitem{Toma}
M. Tomasevic, D. Talay, {\em A new McKean-Vlasov stochastic interpretation of the parabolic-parabolic Keller-Segel model: The one-dimensional case}, 
Bernoulli, 26, 2020.

\bibitem{Ver}
M. Veraar,  {\em The stochastic Fubini theorem revisited}, Stochastics, 84(4), pp. 543–551, 2012.
\end{thebibliography}
\end{document}